\documentclass[11pt]{article}
\usepackage{amssymb}
\usepackage{epic}
\usepackage{authblk}
\begin{document}
\newcommand{\qed}{\hfill$\rule{.05in}{.1in}$ \vspace{.3cm}}
\newcommand{\pf}{\noindent{\bf Proof: }}
\newtheorem{thm}{Theorem}
\newtheorem{lem}{Lemma}
\newtheorem{prop}{Proposition}
\newtheorem{ex}{Example}
\newtheorem{cor}{Corollary}
\newtheorem{conj}{Conjecture}
\newtheorem{prob}{Problem}
\newtheorem{claim}{Claim}
\newcommand{\beq}{\begin{equation}}
\newcommand{\eeq}{\end{equation}}
\newcommand{\<}[1]{\left\langle{#1}\right\rangle}
\newcommand{\be}{\begin{enumerate}}
\newcommand{\ee}{\end{enumerate}}
\newcommand{\al}{\alpha}
\newcommand{\ep}{\epsilon}
\newcommand{\si}{\sigma}
\newcommand{\om}{\omega}
\newcommand{\la}{\lambda}
\newcommand{\La}{\Lambda}
\newcommand{\ga}{\gamma}
\newcommand{\im}{\Rightarrow}
\newcommand{\2}{\vspace{.2cm}}
\newcommand{\es}{\emptyset}
\newcommand{\lz}{\langle}
\newcommand{\rz}{\rangle}
\newcommand{\bh}{\hat}

\newcommand{\ad}{\mathrel{\raisebox{0.2ex}{$\scriptstyle\sim$}}}
\newcommand{\nad}
{\mathrel{\raisebox{0.2ex}{$\scriptstyle\nsim$}}}

\date{15 March 2026}

\title{A Characterization of $P_6$-Free Irredundance\\ Perfect Graphs}
\author[1]{Vadim Zverovich\footnote{Corresponding author (e-mail: vadim.zverovich@uwe.ac.uk)}}
\author[2]{Pavel Skums}
\author[3]{Lutz Volkmann}
\affil[ ]{\footnotesize $^1$Mathematics and Statistics Research Group, University of the West of England, Bristol, UK}
\affil[2]{\footnotesize School of Computing, University of Connecticut, Storrs, Connecticut, USA}
\affil[3]{\footnotesize Department II of Mathematics,
RWTH Aachen,
Aachen 52056, Germany} 

\maketitle

\begin{abstract}
Let $ir(G)$ and $\gamma(G)$ be the irredundance number and the
domination number of a graph $G$, respectively. A graph $G$ is
called {\it irredundance perfect} if $ir(H)=\gamma(H)$ for every
induced subgraph $H$ of $G$. The subclass of $P_6$-free irredundance perfect graphs has been studied extensively. In this paper, we
present a characterization of this graph class in terms of eleven forbidden induced subgraphs.

\vspace*{.3cm} {\footnotesize \noindent Keywords: {\it
irredundance perfect graphs, irredundance number, domination
number, $P_6$-free graphs
}}

\end{abstract}


All graphs considered in this paper are finite, simple, and undirected. For a graph $G$, let $V(G)$ denote its vertex set; for $X \subseteq V(G)$, let $\<X$ be the subgraph of $G$ induced by $X$. For $X, Y \subseteq V(G)$, we write $X \ad Y$ if every vertex of $X$ is adjacent to every vertex of $Y$, and $X \nad Y$ if no vertex of $X$ is adjacent to any vertex of $Y$. In particular, $x \ad y$ means that the vertices $x$ and $y$ are adjacent. For a vertex $x \in V(G)$, let $N(x)$ and $N[x] = N(x)\cup \{x\}$ denote its open and closed neighborhoods, respectively. This notation extends to subsets $X\subseteq V(G)$ by defining $N(X)=\bigcup_{x\in X} N(x)$
and $N[X]=N(X)\cup X$.

A set $X\subseteq V(G)$ {\it dominates} a set $Y\subseteq V(G)$ if
$Y\subseteq N[X]$. In particular, if $X$ dominates $V(G)$, then
$X$ is called a {\it dominating set}.
The {\it domination number} $\ga(G)$ is the
cardinality of a minimum dominating set of $G$.
Domination in graphs has been extensively studied and has numerous real-life applications \cite{Zve3}.

For $x\in X$, the set
$$
PN(x,X) = N[x]-N[X - \{x\}]
$$
is called the {\it private neighborhood} of $x$; the shorter notation $PN(x)$ is used when $X$ is clear from the context.
If $PN(x,X)=\es$, then $x$ is said to be {\it redundant} in $X$.
A set $X$ containing no redundant vertex is called {\it irredundant}.
The minimum cardinality of a maximal
irredundant set of $G$ is the {\it irredundance number} $ir(G)$.
The following proposition provides the necessary and sufficient
condition for an irredundant set to be maximal.

\begin{prop}
[\cite{Vol4}] Let $X$ be an irredundant set of $G$, and
$U=V(G)-N[X]$. The set $X$ is a maximal irredundant set if and
only if for any $v\in N[U]$, the vertex $v$ dominates $PN(x,X)$
for some vertex $x\in X$.
\end{prop}

It is well known that for any graph $G$, $ir(G) \le \gamma(G)$. 
A graph $G$ is \emph{irredundance perfect} if $ir(H)=\gamma(H)$ for every induced subgraph $H$ of $G$. 
There are numerous interesting results and conjectures concerning irredundance perfect graphs 
\mbox{[1--4,} 6--13]; some of them are discussed below. 
Related classes of graphs, such as domination perfect graphs, upper domination perfect graphs, and upper irredundance perfect graphs, have also been studied (see, e.g., \cite{Gut,Zve1,Zve2}).


One line of research on irredundance perfect graphs has focused on their characterization by forbidden induced subgraphs. 
Below we outline some milestones in this direction.
The earliest result was obtained by Bollobás and Cockayne:

\begin{thm}
[Bollob\'as and Cockayne \cite{Bol}]
\label{B1}
If a graph $G$ contains no pair of induced subgraphs isomorphic to $P_4$
with vertex sequences $(a_i,b_i,c_i,d_i)$, $i=1,2$, such that 
$b_1,b_2,c_1,$
$c_2,d_1,d_2$ are pairwise distinct and 
$a_1,a_2 \not \in \{c_1,c_2,d_1,d_2\}$, 
then $G$ is irredundance perfect.
\end{thm}

The following result of Favaron strengthens Theorem~\ref{B1}, as the graphs forbidden in Theorem~\ref{F1} are included among those forbidden in Theorem~\ref{B1}.

\begin{thm}
[Favaron \cite{Fav}]
\label{F1}
If a graph $G$ does not contain the graphs $P_6$, $C_6$, $2P_4$ 
and $G_1-G_3$ (see Figure \ref{5graphs}) as induced subgraphs, then $G$
is an irredundance perfect graph.
\end{thm}

\vspace*{-0.2cm}

\begin{figure}[h!]
\begin{center}
\setlength{\unitlength}{1cm}
\thicklines
\begin{picture}(15,4)
\put(1,1){\circle*{0.2}}
\put(1,2){\circle*{0.2}}
\put(1,3){\circle*{0.2}}
\put(2,1){\circle*{0.2}}
\put(2,2){\circle*{0.2}}
\put(2,3){\circle*{0.2}}
\drawline(1,1)(1,3)(2,3)(2,1)
\drawline(1,2)(2,2)

\put(4,1){\circle*{0.2}}
\put(4,2){\circle*{0.2}}
\put(4,3){\circle*{0.2}}
\put(5,1){\circle*{0.2}}
\put(5,2){\circle*{0.2}}
\put(5,3){\circle*{0.2}}
\drawline(4,1)(4,3)(5,3)(5,1)(4,1)
\drawline(4,2)(5,2)

\put(7,1){\circle*{0.2}}
\put(7,2){\circle*{0.2}}
\put(7,3){\circle*{0.2}}
\put(8,1){\circle*{0.2}}
\put(8,2){\circle*{0.2}}
\put(8,3){\circle*{0.2}}
\put(7.5,3.5){\circle*{0.2}}
\drawline(7,1)(7,3)(7.5,3.5)(8,3)(8,1)
\drawline(7,2)(8,2)

\put(10,1){\circle*{0.2}}
\put(10,2){\circle*{0.2}}
\put(10,3){\circle*{0.2}}
\put(11,1){\circle*{0.2}}
\put(11,2){\circle*{0.2}}
\put(11,3){\circle*{0.2}}
\put(10.5,3.5){\circle*{0.2}}
\drawline(10,1)(10,3)(10.5,3.5)(11,3)(11,1)
\drawline(10,2)(11,2)
\drawline(10,3)(11,3)

\put(13,1){\circle*{0.2}}
\put(13,2){\circle*{0.2}}
\put(13,3){\circle*{0.2}}
\put(14,1){\circle*{0.2}}
\put(14,2){\circle*{0.2}}
\put(14,3){\circle*{0.2}}
\put(13.5,3.5){\circle*{0.2}}
\drawline(13,1)(13,3)(13.5,3.5)(14,3)(14,1)(13,1)
\drawline(13,2)(14,2)
\drawline(13,3)(14,3)

\put(1.3,.2){$G_1$}
\put(4.3,.2){$G_2$}
\put(7.3,.2){$G_3$}
\put(10.3,.2){$G_4$}
\put(13.3,.2){$G_5$}
\end{picture}
\end{center}
\caption{Graphs $G_1-G_5$.}
\label{5graphs}
\end{figure}
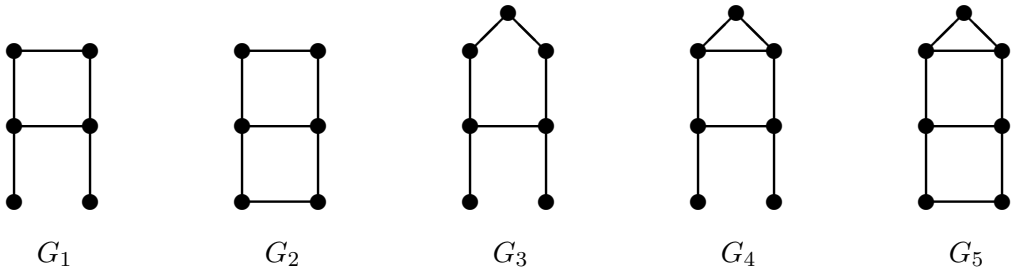

Favaron conjectured that only three of the six graphs in Theorem~\ref{F1} suffice as forbidden induced subgraphs 
to characterize irredundance perfect graphs.

\begin{conj}
[Favaron \cite{Fav}, see also \cite{Cha}]
\label{conj1}
If a graph $G$ does not contain the graphs $P_6$ and $G_1,G_2$ (Figure \ref{5graphs}) as induced subgraphs, then $G$ is irredundance perfect.
\end{conj}

Henning \cite{Hen} proved Conjecture \ref{conj1} for graphs $G$ 
with $ir(H)\le 4$ for every induced subgraph $H$ of $G$; it was later completely proved 
in \cite{Pue,Vol2}.
In fact, \cite{Vol2} provides the stronger result: a graph is irredundance perfect if it does not contain $P_6$, $G_1$ and $G_5$ (Figure \ref{5graphs}) as induced subgraphs.

Another conjecture on irredundance perfect graphs, independent from Conjecture \ref{conj1}, has been posed in \cite{Fau}:

\begin{conj}
[Faudree, Favaron and Li \cite{Fau}]
\label{conj2}
Any $P_5$-free graph is irredundance perfect.
\end{conj}

The validity of Conjecture~\ref{conj2} follows from a result of Puech \cite{Pue}, who showed that a graph is irredundance perfect if it contains neither $P_6$ nor $H$ as an induced subgraph, where $H$ is the graph obtained from $G_4$ by deleting one of its pendant vertices. 
In the same paper, Puech proposed Conjecture \ref{conj3} that generalizes Conjectures~\ref{conj1} and~\ref{conj2}: 

\begin{conj}
[Puech \cite{Pue}]
\label{conj3}
If a graph $G$ does not contain the graphs $P_6$ and $G_4,G_5$ (Figure \ref{5graphs}) as induced subgraphs, then $G$ is irredundance perfect.
\end{conj}

Volkmann and Zverovich proved 
Conjecture \ref{conj3} 
in \cite{Vol5}.

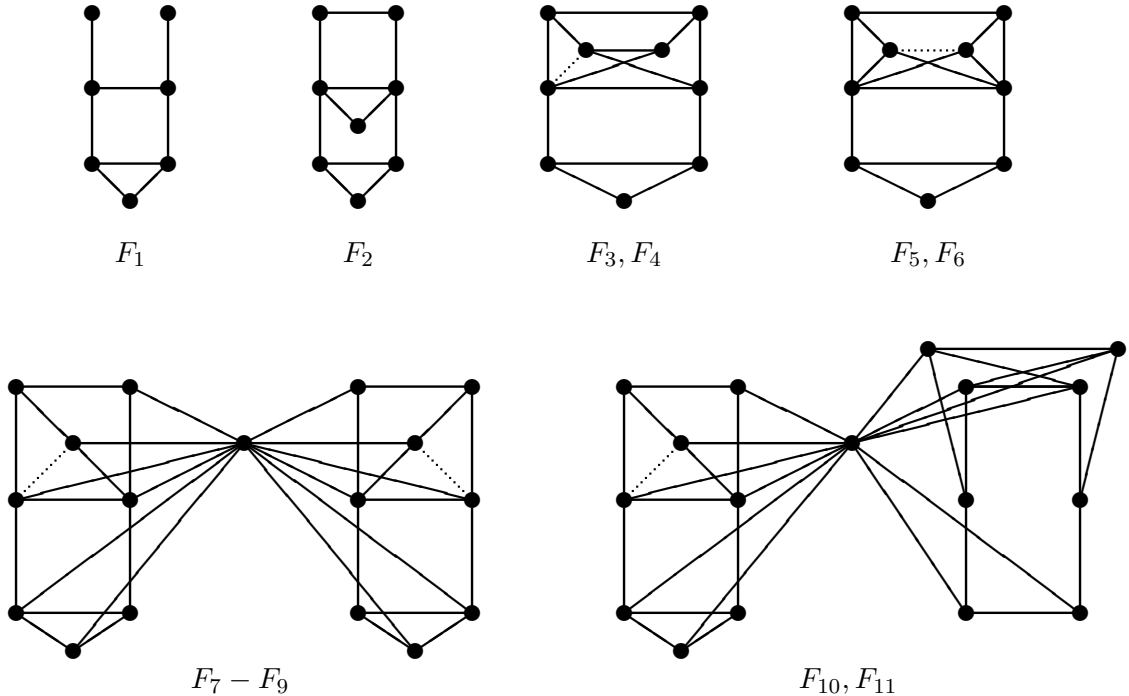
\begin{figure}[b!]
\begin{center}
\setlength{\unitlength}{1cm}
\thicklines
\begin{picture}(14,4.5)
\put(1,2){\circle*{0.2}}
\put(1,3){\circle*{0.2}}
\put(1,4){\circle*{0.2}}
\put(1.5,1.5){\circle*{0.2}}
\put(2,2){\circle*{0.2}}
\put(2,3){\circle*{0.2}}
\put(2,4){\circle*{0.2}}
\drawline(1,4)(1,2)(1.5,1.5)(2,2)(2,4)
\drawline(1,2)(2,2)
\drawline(1,3)(2,3)
\put(1.3,0.7){$F_1$}

\put(4,2){\circle*{0.2}}
\put(4,3){\circle*{0.2}}
\put(4,4){\circle*{0.2}}
\put(4.5,1.5){\circle*{0.2}}
\put(5,2){\circle*{0.2}}
\put(5,3){\circle*{0.2}}
\put(5,4){\circle*{0.2}}
\drawline(4,4)(4,2)(4.5,1.5)(5,2)(5,4)(4,4)
\drawline(4,2)(5,2)
\drawline(4,3)(5,3)
\put(4.5,2.5){\circle*{0.2}}
\drawline(4,3)(4.5,2.5)(5,3)
\put(4.3,.7){$F_2$}

\put(7,2){\circle*{0.2}}
\put(7,3){\circle*{0.2}}
\put(7,4){\circle*{0.2}}
\put(8,1.5){\circle*{0.2}}
\put(9,2){\circle*{0.2}}
\put(9,3){\circle*{0.2}}
\put(9,4){\circle*{0.2}}
\drawline(7,4)(7,2)(8,1.5)(9,2)(9,4)(7,4)
\drawline(7,2)(9,2)
\dottedline{.08}(7.5,3.5)(7,3)
\drawline(7,3)(9,3)
\drawline(8.5,3.5)(7.5,3.5)
\put(7.5,3.5){\circle*{0.2}}
\put(8.5,3.5){\circle*{0.2}}
\drawline(7,3)(8.5,3.5)(9,4)
\drawline(7,4)(7.5,3.5)(9,3)
\put(7.5,.7){$F_3,F_4$}

\put(11,2){\circle*{0.2}}
\put(11,3){\circle*{0.2}}
\put(11,4){\circle*{0.2}}
\put(12,1.5){\circle*{0.2}}
\put(13,2){\circle*{0.2}}
\put(13,3){\circle*{0.2}}
\put(13,4){\circle*{0.2}}
\drawline(11,4)(11,2)(12,1.5)(13,2)(13,4)(11,4)
\drawline(11,2)(13,2)
\drawline(11,3)(13,3)
\dottedline{.08}(12.5,3.5)(11.5,3.5)
\drawline(11,3)(11.5,3.5)
\drawline(13,3)(12.5,3.5)
\put(11.5,3.5){\circle*{0.2}}
\put(12.5,3.5){\circle*{0.2}}
\drawline(11,3)(12.5,3.5)(13,4)
\drawline(11,4)(11.5,3.5)(13,3)
\put(11.5,.7){$F_5,F_6$}
\end{picture}
\end{center}

\begin{center}
\setlength{\unitlength}{1cm}
\thicklines
\begin{picture}(14,4.5)

\put(0,1){\circle*{0.2}}
\put(0,2.5){\circle*{0.2}}
\put(0,4){\circle*{0.2}}
\put(.75,.5){\circle*{0.2}}
\put(.75,3.25){\circle*{0.2}}
\put(1.5,1){\circle*{0.2}}
\put(1.5,2.5){\circle*{0.2}}
\put(1.5,4){\circle*{0.2}}
\put(3,3.25){\circle*{0.2}}
\drawline(0,1)(.75,.5)(1.5,1)(0,1)(0,4)(1.5,4)(1.5,1)
\drawline(0,2.5)(1.5,2.5)(0,4)
\dottedline{.08}(0,2.5)(.75,3.25)
\drawline(0.75,0.5)(3,3.25)(0,1)
\drawline(1.5,2.5)(3,3.25)(0,2.5)
\drawline(0.75,3.25)(3,3.25)(1.5,4)

\put(4.5,1){\circle*{0.2}}
\put(4.5,2.5){\circle*{0.2}}
\put(4.5,4){\circle*{0.2}}
\put(5.25,.5){\circle*{0.2}}
\put(5.25,3.25){\circle*{0.2}}
\put(6,1){\circle*{0.2}}
\put(6,2.5){\circle*{0.2}}
\put(6,4){\circle*{0.2}}
\drawline(4.5,1)(5.25,.5)(6,1)(4.5,1)(4.5,4)(6,4)(6,1)
\drawline(6,2.5)(4.5,2.5)(6,4)
\dottedline{.08}(6,2.5)(5.25,3.25)
\drawline(5.25,0.5)(3,3.25)(6,1)
\drawline(4.5,2.5)(3,3.25)(6,2.5)
\drawline(5.25,3.25)(3,3.25)(4.5,4)
\put(2.3,0){$F_7-F_9$}

\put(8,1){\circle*{0.2}}
\put(8,2.5){\circle*{0.2}}
\put(8,4){\circle*{0.2}}
\put(8.75,.5){\circle*{0.2}}
\put(8.75,3.25){\circle*{0.2}}
\put(9.5,1){\circle*{0.2}}
\put(9.5,2.5){\circle*{0.2}}
\put(9.5,4){\circle*{0.2}}
\put(11,3.25){\circle*{0.2}}
\drawline(8,1)(8.75,.5)(9.5,1)(8,1)(8,4)(9.5,4)(9.5,1)
\drawline(8,2.5)(9.5,2.5)(8,4)
\dottedline{.08}(8,2.5)(8.75,3.25)
\drawline(8.75,0.5)(11,3.25)(8,1)
\drawline(9.5,2.5)(11,3.25)(8,2.5)
\drawline(8.75,3.25)(11,3.25)(9.5,4)

\put(12.5,1){\circle*{0.2}}
\put(12.5,2.5){\circle*{0.2}}
\put(12.5,4){\circle*{0.2}}
\put(12,4.5){\circle*{0.2}}
\put(14.5,4.5){\circle*{0.2}}
\put(14,1){\circle*{0.2}}
\put(14,2.5){\circle*{0.2}}
\put(14,4){\circle*{0.2}}
\drawline(14,2.5)(14,1)(12.5,1)(12.5,4)(14,4)(14,2.5)(14.5,4.5)(12,4.5)(12.5,2.5)
\drawline(12,4.5)(14,4)(12.5,4)(14.5,4.5)
\drawline(12.5,1)(11,3.25)(14,1)
\drawline(12.5,4)(11,3.25)(14,4)
\drawline(12,4.5)(11,3.25)(14.5,4.5)
\put(10.3,0){$F_{10},F_{11}$}
\end{picture}
\end{center}
\caption{Graphs $F_1-F_{11}$. Dotted lines denote additional edges that should be added to produce the next graph in the corresponding list.}
\label{11graphs}
\end{figure}

The main result of this paper is  the complete characterization of $P_6$-free irredundance perfect graphs in terms of forbidden induced subgraphs 
given in Theorem \ref{t1}. 
All results discussed above follow  from Theorem \ref{t1}.

\begin{thm}
\label{t1}
A $P_6$-free graph $G$ is irredundance perfect if and only if $G$ does
not contain the graphs $F_1-F_{11}$ in Figure \ref{11graphs} as induced subgraphs.
\end{thm}

\pf
The 'only if' part is trivial, since $ir(F_i)<\ga(F_i)$ for every
graph $F_i$ of Figure \ref{11graphs}.
To prove the 'if' part of the theorem,
let $H$ be a minimum counterexample, i.e., $H$ does not contain
$P_6$, $F_1-F_{11}$ as induced subgraphs, $ir(H)<\gamma(H)$
and $ir(H')=\gamma(H')$ for every proper induced subgraph $H'$ of $H$.
Let $X$ be a maximal irredundant
set of $H$ of cardinality $ir(H)$.
Denote
$$U=V(H)-N[X],$$
$$PN=\cup_{x\in X}PN(x,X),$$
$$W=V(H)-X-PN-U.$$
Observe that $|N(w)\cap X|\ge 2$ for any vertex $w\in W$.
Also, $U\not= \es$ because $X$ does not dominate $H$.

The proof of Theorem \ref{t1} is based on the next lemmas.

\begin{lem}
\label{l1} The graph $\<X$ has no vertex of degree 0.
\end{lem}

\pf Let $v$ be an isolated vertex in $\<X$. Denote $H'=H-N[v]$ and
$X'=X-\{v\}$. Clearly, for any vertex $x\in X'$,
$$
PN_{H'}(x,X')=PN_H(x,X)\not=\es.
$$
Therefore, $X'$ is an irredundant set in $H'$. Suppose that there
is a vertex $u\in V(H')-X'$ such that the set $X'\cup \{u\}$ is
irredundant in $H'$, i.e., $PN_{H'}(x,X'\cup \{u\})\not= \es$ for
any vertex $x\in X'\cup\{u\}$. It is not difficult to see that
$$
PN_H(x,X\cup \{u\})=PN_{H'}(x,X'\cup \{u\})\not=\es
$$
for any $x\in X'\cup \{u\}$. Moreover, $PN_H(v,X\cup \{u\})\not=
\es$, since $v\in PN_H(v,X\cup \{u\})$. We conclude that $X\cup
\{u\}$ is an irredundant  set in $H$, contrary to the fact that
$X$ is maximal irredundant. Consequently, $X'$ is a maximal
irredundant set in $H'$ and hence $ir(H')\le |X'|=ir(H)-1$. If $D$
is a minimum dominating set of $H'$, then $D\cup \{v\}$ is a
dominating set of $H$. Therefore, $\gamma(H)\le
|D|+1=\gamma(H')+1$. We obtain
$$ir(H')\le ir(H)-1 < \gamma(H)-1\le \gamma(H').$$
Thus, $ir(H')<\gamma(H')$, contrary to the minimality of the graph
$H$. \qed

By Proposition 1, for every vertex $u\in U$ there is a vertex $x\in X$
such that $PN(x,X)\subseteq N(u)$.
Define
$$X_u=\{x\in X: PN(x,X)\subseteq N(u)\}.$$
Let $F$ denote a subset of $X$ of smallest cardinality such that
$X_u\cap F\not= \es$ for each $u\in U$.

\begin{lem}
\label{l2}
If $f_1,f_2\in F$, then there exist
vertices $y_i\in PN(f_i)$ and $u_i\in U$ $(i=1,2)$ such that
$u_1\ad y_1, u_2\ad y_2$ and $u_1\nad y_2, u_2\nad y_1$.
Moreover, for $i=1,2$, $u_i$ dominates $PN(f_i)$ and  $u_i$
does not dominate $PN(f)$ for each $f\in F-f_i$.
\end{lem}

\pf
The minimality of $F$ implies that for $f_1\in F$ there is a $u_1\in U$
such that $u_1$ dominates $PN(f_1)$ and $X_u\cap (F-f_1)=\es$, i.e.,
$u_1$ does not dominate $PN(f)$ for any $f\in F-f_1$.
Analogously, for $f_2\in F$ there is a $u_2\in U$ dominating $PN(f_2)$
and not dominating $PN(f)$ for any $f\in F-f_2$.
Choosing $y_1\in PN(f_1)$ so that $y_1\nad u_2$ and choosing
$y_2\in PN(f_2)$ so that $y_2\nad u_1$, we obtain the desired result.
\qed

Put $Z=X-F$, thus
$X=F\cup Z.$
If $L$ is an arbitrary maximal connected component of the graph $H-W$,
then define
$$
L_X=V(L)\cap X, \quad
L_F=L_X\cap F, \quad
L_Z=L_X\cap Z,
$$
$$
L_{PN}=V(L)\cap PN, \qquad
L_U=V(L)\cap U.
$$
Observe that 
$$
V(L)=L_X\cup L_{PN}\cup L_U.
$$
By Proposition 1, any vertex of $L_U$ is adjacent to some vertex of $L_{PN}$.
Also, any vertex of $L_{PN}$ is adjacent to some vertex of $L_X$.
Therefore, $L_X\not=\es$.

\begin{lem}
\label{l3}
The graph $\lz  L_X\rz$ is connected.
\end{lem}

\pf Suppose to the contrary that $\lz  L_X\rz$ is not connected,
and let $L_1$ and $L_2$ be two connected components of $\lz
L_X\rz$. Consider the shortest path $P$ between $L_1$ and $L_2$ in
$L$. The path $P$ has the form $P=(x_1,v_1,...,v_k,x_2)$, where
$x_1\in L_1$ and $x_2\in L_2$. Observe that $v_1\not\in L_1$,
$v_1\in PN(x_1,X)$ and $v_k\not\in L_2$, $v_k\in PN(x_2,X)$, and
hence $v_1\not=v_k$ and $k\ge 2$. By Lemma \ref{l1}, $x_1\ad x'_1$
and $x_2\ad x'_2$, where $x'_1,x'_2\in X$. Clearly, $x'_1,x'_2\in
V(L)$ and hence $x'_1\in L_1$ and $x'_2\in L_2$. Note that
$x'_1\nad v_1$, since $v_1\in PN(x_1)$, and also $x'_1\nad v_i$
for $i>1$, for otherwise the path $P$ is not shortest.
Analogously, $x'_2\nad v_i$ for any $i$, $1\le i\le k$. Now, the
path $(x'_1,x_1,v_1,...,v_k,x_2,x'_2)$ contains the
path $P_6$ as an induced subgraph, a contradiction. Thus, $\lz
L_X\rz$ is a connected graph. \qed

Now we show that the component $L$ may have only one of basic
types defined below. Firstly, suppose that the set $L_F$ is
independent. If $L_Z=\es$, then we obtain a contradiction 
with Lemma
\ref{l1} or \ref{l3}. Hence $L_Z\not=\es$. Let there exist a
vertex $z\in L_Z$ such that $|N(z)\cap L_F|\ge 2$, in this case
$L$ is called a {\it component of type A.} Let $f_1,f_2\in
N(z)\cap L_F$ and let $y_1,y_2,u_1,u_2$ be chosen as in Lemma
\ref{l2}. We have $y_1\ad y_2$, for otherwise $\lz
u_1,y_1,f_1,z,f_2,y_2\rz\cong P_6$. Let us show that $v\ad f_1$
for an arbitrary vertex $v\in L_X-\{z,f_1,f_2\}$. Suppose that
$v\nad f_1$ and consider the shortest $(f_1-v)$-path $P$ in $\lz
L_X\rz$. Such a path must exist by Lemma \ref{l3}. The distance
between $f_1$ and $f_2$ in $\lz  L_X\rz$ is two. So, if $f_2\in
P$, then $P$ has the form $P=(f_1,z',f_2,v_1,...,v_k=v),$ where
$k\ge 1$. We obtain $\lz u_1,y_1,f_1,z',f_2,v_1\rz\cong P_6$.
Hence $f_2\not\in P$ and $P$ has the form $P=(f_1,v_1,...,v_k=v),$
where $k\ge 2$. We have $\lz u_2,y_2,y_1,f_1,v_1,v_2\rz\cong P_6$.
Thus, we have proved that $v\ad f_1$. We can show analogously that
$v\ad f_2$. Since $L_F$ is independent, we obtain $v\in L_Z$, and
hence $L_F=\{f_1,f_2\}$. Thus, $z\ad \{f_1,f_2\}$ for any vertex
$z\in L_Z$. We summarize all the facts about the component $L$ of
type A.

{\bf Type A.} $L_F=\{f_1,f_2\}$, $f_1\nad f_2$, $L_Z\not=\es$ and
$z\ad \{f_1,f_2\}$ for any vertex $z\in L_Z$.
Further, let $y_i,u_i$ $(i=1,2)$ be chosen by Lemma \ref{l2},
i.e., $u_i\ad y_i$, $y_i\ad f_i$ $(i=1,2)$ and $u_1\nad y_2$, $u_2\nad y_1$.
Also, $y_1\ad y_2$.
Put $\bh N = N_1\cup N_2\cup N_{1,2}$, where
\begin{eqnarray}
N_1 &=& \{u\in L_U: u\ad y_1, u\nad y_2\}, \nonumber \\
N_2 &=& \{u\in L_U: u\nad y_1, u\ad y_2\}, \\
N_{1,2} &=& \{u\in L_U: u\ad y_1, u\ad y_2\}, \nonumber
\end{eqnarray}

\begin{lem}
\label{l4}
Let $L$ be a component of type A. Then

(a) The set $\{y_1,y_2\}$ dominates $PN(f_1)\cup PN(f_2)\cup L_U$;

(b) $N_1\nad N_2$, $N_1\nad N_{1,2}$, and $N_2\nad N_{1,2}$;

(c) $L_U=N_1\cup N_2\cup N_{1,2}.$
\end{lem}

\pf Suppose that there is a vertex $y$, say $y\in PN(f_1)$, such
that $y\nad \{y_1,y_2\}$. By Lemma \ref{l2}, $y\ad u_1$. We obtain
$\lz y,u_1,y_1,y_2,f_2,z\rz\cong P_6$, where $z\in L_Z$, a
contradiction. If there is a vertex $u\in L_U$ such that $u\nad
\{y_1,y_2\}$, then $f_1,f_2\not\in X_u$ and hence
$$X_u\subseteq L_Z\subseteq Z.$$
Thus, $X_u\cap F=\es$, contrary to the choice of
$F$. To prove the second statement of the lemma, let $v\in N_1$ be
adjacent to $u\in N_2\cup N_{1,2}$. We have $\lz
v,u,y_2,f_2,z,f_1\rz\cong P_6$, where $z\in L_Z$, a contradiction.
If $v\in N_2$ is adjacent to $u\in N_{1,2}$, then $\lz
v,u,y_1,f_1,z,f_1\rz\cong P_6$, a contradiction. Since
$\{y_1,y_2\}$ dominates $L_U$, we obtain $L_U=N_1\cup N_2\cup
N_{1,2}$. \qed

Now, suppose that there is a $z\in L_Z$ such that $N(z)\cap
L_F=\{f\}$. In this case, we say that the {\it component $L$ has
type B.} Assume that $|L_F|>1$, i.e., there is an $f'\in L_F$,
$f'\not=f$. Further, consider the shortest $(f-f')$-path $P$ in
$\lz  L_X\rz$. Such a path must exist by Lemma \ref{l3}. Recall
that $f\nad f'$, since $L_F$ is independent. If $P=(f,g,f')$, then
$g$ is adjacent to $f,f'\in L_F$ and $g\in L_Z$ because $L_F$ is
independent. Therefore, $|N(g)\cap L_F|\ge 2$, i.e., $L$ must be
the component of type A, a contradiction. Thus,
$P=(f,z_1,z_2,...,z_k=f')$, where $k\ge 3$. Since $f\in F$, there
is a $u\in U$ such that $u\ad y\in PN(f)$. We obtain $\lz
u,y,f,z_1,z_2,z_3\rz\cong P_6$. Thus, $|L_F|=1$. Summarize the
above facts.

{\bf Type B.} $L_F=\{f\}$ and $L_Z\not=\es$.
Let $y\in PN(f)$.
The vertex $y$ dominates $L_U$.
If, to the contrary, there is a $u\in L_U$ such that $u\nad y$,
then $f\not\in X_u$ and hence $X_u\subseteq L_Z\subseteq Z$,
i.e., $X_u\cap F=\es$, a contradiction.

Consider now the case $N(z)\cap L_F=\es$ for any $z\in L_Z$.
By Lemma \ref{l3}, $\lz L_X\rz$ is connected, and so $L_F=\es$.
The component of this type is called a {\it component of type C.}

{\bf Type C.} $L_F=\es$ and $L_Z=L_X\not=\es$.
Also, $L_U=\es$, for otherwise $\es\not=X_u\subseteq  L_Z\subseteq Z$ for
$u\in L_U$, a contradiction.

Now, let $L_F$ be a dependent set, i.e.,
there are $f_1,f_2\in L_F$ with $f_1\ad f_2$.
If $|L_X|=2$, then we say that $L$ {\it has type D.}

{\bf Type D.} $L_F=\{f_1,f_2\}$, $f_1\ad f_2$, and $L_Z=\es$.
Let $y_i,u_i$ $(i=1,2)$ be chosen by Lemma \ref{l2}, i.e.,
$u_i\ad y_i$, $y_i\ad f_i$ $(i=1,2)$ and $u_1\nad y_2$, $u_2\nad y_1$.
Moreover, $PN(f_i)\subseteq N(u_i)$ for $i=1,2$.
The set $\{y_1,y_2\}$ dominates $L$.
Indeed, any vertex $v$ of $PN(f_1)\cup PN(f_2)\cup L_U$
belongs to $N[L_U]$ and, by Proposition 1, $v$
must dominate $PN(f_1)$ or $PN(f_2)$, i.e., $v$ is adjacent to $y_1$ or
$y_2$.

Consider the case $|L_X|\ge 3$ and suppose firstly that
$L_X$ has no $K_3=\lz f',f'',x\rz$ where $f',f''\in L_F$ and $x\in L_X$.
For this case, the components of types E, F, and G will be defined.

\begin{lem}
\label{l5}
Let $|L_X|\ge 3$, $f_1,f_2\in F$, $f_1\ad f_2$,
let $L_X$ have no $K_3=\lz f',f'',x\rz$ where $f',f''\in L_F$ and $x\in L_X$,
and let $y_1,y_2,u_1,u_2$ be chosen by Lemma \ref{l2}.
Then

(a) $y_1\ad y_2$ and $u_1\nad u_2$;

(b) $\lz L_X\rz$ consists of two disjoint graphs $R_1$ and $R_2$
connected by the bridge $f_1f_2$;

(c) $f_1$ dominates $R_1$, and $f_2$ dominates $R_2$;

(d) $\{u_1,y_1\}$ dominates $\cup_{x\in V(R_2)-f_2}PN(x)$,
$\{u_2,y_2\}$ dominates $\cup_{x\in V(R_1)-f_1}PN(x)$.
\end{lem}

\pf
By Lemma \ref{l3}, $\lz L_X\rz$ is a connected graph.
Since $|L_X|\ge 3$, there is an $x\in L_X-\{f_1,f_2\}$ adjacent to $f_1$ or $f_2$,
say  $x\ad f_1$ and $x\nad f_2$.
If $u_1\ad u_2$, then $\lz x,f_1,f_2,y_2,u_2,u_1\rz\cong P_6$,
a contradiction.
Consequently, $u_1\nad u_2$.
We have $y_1\ad y_2$, for otherwise $\lz u_1,y_1,f_1,f_2,$
$y_2,u_2\rz\cong P_6$.
Let $R_1$ be a connected component of the graph $R-f_2$ containing $f_1$.
If $f_1$ does not dominate $R_1$, then there are $x_1,x_2$ from $R_1$
such that $\lz u_2,y_2,y_1,f_1,x_1,x_2\rz\cong P_6$.
Therefore, $f_1$ dominates $R_1$.
In the same way, $f_2$ dominates the connected component $R_2$ of $R-f_1$
containing $f_2$.
Thus, $f_1f_2$ is a bridge in $\lz L_X\rz$, since no vertex of $\lz L_X\rz$
can be adjacent to both $f_1$ and $f_2$.
If $\{u_2,y_2\}$ does not dominate $PN(y)$ for some $y\in V(R_1)-f_1$,
then there is a $g\in PN(X)$ such that $g\nad \{u_2,y_2\}$ and
$\lz u_2,y_2,f_2,f_1,y,g\rz\cong P_6$, a contradiction.
In the same way, we can show that
$\{u_1,y_1\}$ dominates $PN(y)$ for any $y\in V(R_2)-f_2$.
\qed

If $|R_1|\ge 2$ and $|R_2|\ge 2$, then we obtain a {\it component
of type E.} Suppose that there is an $f\in L_F-\{f_1,f_2\}$, say
$f\in R_1$. By Lemma \ref{l5}, $f\ad f_1$ and $x\nad \{f,f_1\}$
for $x\in R_2-f_2$. Applying Lemma \ref{l5} to $\{f,f_1\}$, we
obtain that $x$ must be adjacent to $f$ or $f_1$, a contradiction.
Thus, $L_F=\{f_1,f_2\}$.

{\bf Type E.} $|R_1|\ge 2$, $|R_2|\ge 2$, $L_F=\{f_1,f_2\}$, $f_1\ad f_2$.
By Lemma \ref{l5}, $y_1\ad y_2$ and $u_1\nad u_2$.
The set $\{y_1,y_2\}$ dominates $PN(z)$ for any $z\in L_Z$.
Indeed, if $u_1\ad g$ for $g\in PN(z)$ and $z\in R_2-f_2$,
then $\lz u_1,g,z,f_2,f_1,z'\rz\cong P_6$, where $z'\in R_1-f_1$,
and hence $u_1\nad PN(z)$ for any $z\in R_2-f_2$.
On the other hand, by Lemma \ref{l5}, $\{u_1,y_1\}$ dominates $PN(z)$
for any $z\in R_2-f_2$. Therefore, $y_1$ dominates $PN(z)$ for any
$z\in R_2-f_2$.
Analogously, $y_2$ dominates $PN(z)$ for any $z\in R_1-f_1$.
Also, $\{y_1,y_2\}$ dominates $L_U$.
Let $z_1\in R_1-f_1$ and $z_2\in R_2-f_2$.
By Lemma \ref{l5}, $\{f_1,f_2\}$ dominates $L_X$.
Also, $\{f_1,f_2\}$ dominates $PN(f_1)\cup PN(f_2)$.
Thus, the set $D=(L_X-\{z_1,z_2\})\cup \{y_1,y_2\}$ dominates $L$
and $|D|=|L_X|$.

Consider the case when only one of the sets $R_1$ and
$R_2$ is of size 1. W.l.o.g., let $|R_1|=1$ and $|R_2|\ge 2$.
This case is subdivided into two subcases.
The {\it type F} is defined by the property $L_F=\{f_1,f_2\}$.

{\bf Type F.} $|R_1|=1$, $|R_2|\ge 2$, $L_F=\{f_1,f_2\}$, $f_1\ad f_2$.
By Lemma \ref{l5}, $y_1\ad y_2$, \mbox{$u_1\nad u_2$,}
$\{u_1,y_1\}$ dominates $PN(z)$ for any $z\in L_Z$,
$f_2$ dominates $L_X$, and $\deg_{\lz L_X\rz}f_1=1$.
Obviously, $\{y_1,y_2\}$ dominates $L_U$.
If there is a $g\in PN(f_1)$ such that $g\nad \{y_1,y_2\}$,
then $\lz g,u_1,y_1,y_2,f_2,z\rz\cong P_6$, where $z\in L_Z$.
Therefore, $\{y_1,y_2\}$ dominates $PN(f_1)\cup L_U$.

Finally, let $|L_F|\ge 3$, $f_1,f_2,f_3\in L_F$.
By Lemma \ref{l5}, $y_1\ad y_2$, $u_1\nad u_2$,
$f_2$ dominates $L_X$, and $\deg_{\lz L_X\rz}f_1=1$.
Suppose that $|L_X|\ge 4$ and let $x\in L_x-\{f_1,f_2,f_3\}$.
Assume that $f_3\ad x$.
Since $f_2$ dominates $L_X$, we have $\lz x,f_2,f_3\rz\cong K_3$,
contrary to the hypothesis, and hence $f_3\nad x$.
By Lemma \ref{l2}, $u_1$ does not dominate $PN(f_3)$ and let $y_3\in PN(f_3)$,
$y_3\nad u_1$.
We have $y_1\ad y_3$, for otherwise $\lz u_1,y_1,f_1,f_2,f_3,y_3\rz\cong P_6$.
Now, $\lz u_1,y_1,y_3,f_3,f_2,x\rz\cong P_6$, a contradiction.
Therefore $|L_X|=3$.

{\bf Type G.} $|R_1|=1$, $|R_2|=2$, $L_X=L_F=\{f_1,f_2,f_3\}$,
$L_Z=\es$, $f_1\ad f_2$.
By Lemma \ref{l5}, $y_1\ad y_2$, $u_1\nad u_2$, $f_2\ad f_3$, $f_1\nad f_3$.
If there is a $g\in PN(f_1)$ such that $g\nad \{y_1,y_2\}$,
then $\lz g,u_1,y_1,y_2,f_2,f_3\rz\cong P_6$.
Therefore, $\{y_1,y_2\}$ dominates $PN(f_1)$.

It remains to consider the case when $|L_X|\ge 3$ and there is an
$x\in L_X$ with $x\ad \{f_1,f_2\}$.
The following lemma will be needed.

\begin{lem}
\label{l6}
Let $f_1,f_2\in F$, $f_1\ad f_2$, and $y_1,y_2,u_1,u_2$ be chosen by Lemma 2.
If there is a vertex $v\in V(H)$ such that $v\ad \{f_1,f_2\}$
and $v\nad \{y_1,y_2\}\cup \bh N$, then $y_1\ad y_2$ and at least
one of the following properties holds:

(i) $u$ dominates $\bh N$ for any $u\in N_1$;

(ii) $u$ dominates $\bh N$ for any $u\in N_2$.
\end{lem}

\pf Suppose that $y_1\nad y_2$. We have $u_1\ad u_2$, for
otherwise $\lz u_1,y_1,f_1,f_2,y_2,u_2\rz\cong P_6$. Now $\lz
v,f_1,y_1,u_1,u_2,y_2\rz\cong P_6$, a contradiction. Thus, $y_1\ad
y_2$. We have $u'\ad u''$ for every $u'\in N_1$ and every $u''\in
N_2$, for otherwise $\lz v,f_1,f_2,y_1,y_2,u',u''\rz\cong F_1$.
Suppose that neither (i) nor (ii) holds, thus there is an $a\in
N_1$ not dominating $\bh N$ and $b\in N_2$ not dominating $\bh N$.
We have $a\ad b$. Let $a\nad c\in \bh N$ and let $b\nad d\in \bh
N$. The vertex $c$ cannot belong to $N_2$, and hence $c\in N_1\cup
N_{1,2}$, i.e., $c\ad y_1$. Also, $d$ cannot belong to $N_1$, and
so $d\in N_2\cup N_{1,2}$ and $d\ad y_2$. Suppose that $c\nad b$.
Then $c\not\in N_1$, i.e., $c\in N_{1,2}$ and hence $c\ad y_2$. We
obtain $\lz v,f_1,f_2,y_1,y_2,c,a,b\rz\cong F_2$. Thus, $c\ad b$.
Analogously, $d\ad a$. Consider the graph $F=\lz
v,f_1,f_2,y_1,y_2,a,b,c,d\rz$. The only edges undetermined in $F$
are $y_1d$, $dc$, and $cy_2$. Suppose that $d\ad c$. If $c\nad
y_2$ then $F\cong F_3$ or $F_4$, and if $c\ad y_2$ then $F\cong
F_4$ or $F_6$. Thus, $d\nad c$. If $d\ad y_1$ and $c\ad y_2$, then
$F\cong F_5$, hence w.l.o.g. $c\nad y_2$. We have $d\ad y_1$, for
otherwise $F-\{a,b\}\cong F_1$. Now $F-a\cong F_2$, a
contradiction. \qed

{\bf Type H.} $|L_X|\ge 3$, $f_1,f_2\in L_F$, $x\in L_X$, $\lz f_1,f_2,x\rz\cong K_3$.
Let $v$ be an arbitrary vertex of $L_X-f_1$.
By Lemma \ref{l3}, $\lz L_X\rz$ is a connected graph, and hence
there is a shortest path $P$ connecting $f_1$ and $v$.
If $|P|\ge 3$, then $\lz u_2,u_1,y_1,f_1\rz\cup P$ contains
the induced $P_6$.
Therefore, $|P|=2$, i.e., $f_1$ is adjacent to each vertex of $L_X-f_1$.
Analogously, $f_2$ dominates $L_X$.
Now let $f\in L_F-\{f_1,f_2\}$.
We have $f\ad f_1$ and, by Lemma \ref{l2}, there are corresponding
vertices $y,y'_1,u,u'_1$.
In the same way as above, we can show that $f$ dominates $L_X$.
Hence $\lz  L_F\rz$ is a complete graph.
By Lemma \ref{l6}, let w.l.o.g. $u_2\in N_2$ dominate $\bh N$.
Suppose that $u_2$ does not dominate $L_U$, i.e., there is a $u\in L_U$
such that $u\nad u_2$.
If $u\ad y_1$ or $u\ad y_2$, then $u\in \bh N$ and hence $u_2\ad u$, a contradiction.
Therefore, $u\nad \{y_1,y_2\}$ and, by Proposition 1, $u$ dominates $PN(f)$ for some $f\in L_F-\{f_1,f_2\}$.
By Lemma \ref{l2}, $u_2$ does not dominate $PN(f)$, and
let $y\in PN(f)$, $y\nad u_2$.
Thus, $u\ad y$, $u\nad \{u_2,y_2\}$, $u_2\nad y$,
and $\lz f_1,f_2,f\rz\cong K_3$, since $\lz L_F\rz$ is a complete graph.
We have $y_1\ad y_2$, for otherwise $\lz u,y,f,f_2,y_2,u_2\rz\cong P_6$.
Now $\lz f_1,f_2,f,y_2,y,u_2,u\rz\cong F_1$, a contradiction.
Thus, $u_2$ dominates $L_U\cup PN(f_2)$.
Put $D=\{u_2\}\cup (L_X-f_2)$.
Obviously, $D$ dominates $L$ and $|D|=|L_X|$.

\2
In fact, we have proved the following lemma.

\begin{lem}
\label{l7}
Any connected component of the graph $H-W$ has one of the types A--H.
\end{lem}

Now we will construct a set $I$ such that
$|I|\le |X|$ and $I$ dominates the graph $H-W$.
If $L$ has type C, then we put $L_Z\subseteq  I$.
Observe that $L_Z$ dominates $L$ and $|L_Z|=|L_X|$.
For the types E and H, we add in $I$ the set $D$ constructed
in the definitions of these types.
The types A, B, D, F, and G are subdivided into subtypes.
For each subtype we define a dominating set $D$ and add it in the set $I$.
Let the component $L$ be of type A.

{\bf Type A1.}
There is a $w\in W$ dominating $L_U\cup PN(f_2)$.
Put $D=(L_X-f_2)\cup \{w\}\subseteq I$.
It is clear that $D$ dominates $L$ and $|D|=|L_X|$.

{\bf Type A2.} There is no $w\in W$ dominating $L_U\cup PN(f_2)$,
and the set $\{y_1,y_2\}$ does not dominate $PN(z)$ for some $z\in
L_Z$. That is, there is a $p\in PN(z)$ and $p\nad \{y_1,y_2\}$. If
there is a $u\in N_1$ such that $u\nad p$, then $\lz
u,y_1,y_2,f_2,z,p\rz\cong P_6$. Therefore, $p$ dominates $N_1$. If
there is a $u\in N_2$ such that $u\nad p$, then $\lz
u,y_2,y_1,f_1,z,p\rz\cong P_6$. Therefore, $p$ dominates $N_2$. At
last, let there exist an $u\in N_{1,2}$ such that $u\nad p$. By
Lemma \ref{l4}, $u\nad u_1$ where $u_1\in N_1$. We have $\lz
u_1,p,z,f_2,y_2,u\rz\cong P_6$, and hence $p$ dominates $N_{1,2}$.
Thus, $p$ dominates $L_U$. Assume that there is a $z'\in L_Z$,
$z'\not=z$. We know that $z'\ad \{f_1,f_2\}$, and $u_1\nad u_2$ by
Lemma \ref{l4}. We obtain $\lz u_1,p,u_2,y_2,f_2,z'\rz\cong P_6$,
and hence $L_Z=\{z\}$. Since $p\nad\{y_1,y_2\}$ and $p\ad u_1$, we
have $p$ dominates $PN(z)$ by Proposition 1. Put
$D=\{f_1,f_2,p\}\subseteq I$. 
Note that $D$ dominates $L$ and
$|D|=|L_X|=3$.

{\bf Type A3.}
There is no $w\in W$ dominating $L_U\cup PN(f_2)$,
and the set $\{y_1,y_2\}$ dominates $\cup_{z\in L_Z}PN(z)$.
Let there exist a vertex $w\in W$ such that $w$ dominates
$L_U\cup PN(z)$ for some $z\in L_Z$ and $w\nad \{f_1,y_1,y_2\}$.
Put $D=(L_X-z)\cup \{w\}\subseteq I$.
It is clear that $D$ dominates $L$ and $|D|=|L_X|$.

{\bf Type A4.}
There is no $w\in W$ dominating $L_U\cup PN(f_2)$,
the set $\{y_1,y_2\}$ dominates $\cup_{z\in L_Z}PN(z)$,
and there is no vertex $w\in W$ such that $w$ dominates $L_U\cup PN(z)$
for some $z\in L_Z$ and $w\nad \{f_1,y_1,y_2\}$.
Put $D=(L_X-\{z,f_2\})\cup\{y_1,y_2\}\subseteq I$, where $z\in L_Z$.
Since $\{y_1,y_2\}$ dominates $L_U\cup L_{PN}$ and $z\ad f_1$,
$f_2\ad y_2$, it follows that $D$ dominates $L$ and $|D|=|L_X|$.

Now let $L$ have type B.

{\bf Type B1.} The vertex $y$ dominates $PN(z)$ for some $z\in L_Z$.
Let $p\in PN(z)$.
Put $D=(L_X-z)\cup \{y\}\subseteq I$.
We have, $y$ dominates 
$L_U\cup PN(z)$,
$L_X-z$ dominates \mbox{$(L_X\cup L_{PN})-(PN(z)\cup \{z\})$.}
By Lemma \ref{l3}, $z\ad v$ for some $v\in L_X$,
and so $L_X-\{z\}$ also dominates $z$.
Thus, $D$ dominates $L$ and $|D|=|L_X|$.

Consider the case when $y$ does not dominate $PN(z)$ for any $z\in L_Z$.
By Lemma \ref{l3}, $f\ad z$ for some $z\in L_Z$.
Denote the set of such components by ${\cal B}$,
and ${\cal D}$ will denote the set of components having type $D$.

{\bf Types B2--B6, D1, D2} are defined by the following algorithmic
procedure running until ${\cal B}=\es$.
\be
\item
Take a component $L\in {\cal B}$.
If ${\cal B}=\es$, then exit.

\item
If there are another component $L'\in {\cal B}$ with the corresponding
vertices $f',y',z'$ and a vertex $w\in W$ such that $w$ dominates
$L_U\cup L'_U\cup PN(z)\cup PN(z')\cup \{f,f'\}$ and $w\nad \{y,z,y',z'\}$, then $L$ is
called a {\it component of type B2} and $L'$ is called a {\it component of type B3}.
We put ${\cal B}={\cal B}-\{L,L'\}$ and start this procedure again.

\item
If there are an $L'\in {\cal D}$ and a $w\in W$ such that $w$ dominates
$L_U\cup L'_U\cup PN(z)\cup \{f,f_1,f_2\}$ and $w\nad \{y,z,y_1,y_2\}$, then
$L$ has {\it type B4} and $L'$ has {\it type D1}.
We put ${\cal B}={\cal B}-L$, ${\cal D}={\cal D}-L'$ and start the procedure
again.

\item
If there are an $L'\in {\cal D}$ and a $w\in W$ such that $w$ dominates
$L_U\cup PN(z)\cup \{f\}$, $w\nad \{y,z\}$, and also
$w\ad L'_{PN}\cup N'_2\cup \{f_1\}$, $w\nad \{u_1,f_2\}$
(or $w\ad L'_{PN}\cup N'_1\cup \{f_2\}$, $w\nad \{u_2,f_1\}$),
then $L$ has {\it type B5} and $L'$ has {\it type D2}.
We put ${\cal B}={\cal B}-L$, ${\cal D}={\cal D}-L'$
and start the procedure again.

\item
If there is no component $L'$ as above, then $L$ is called a {\it component
of type B6}.
We put ${\cal B}={\cal B}-L$ and go on with the procedure.
\ee

If $L$ has type B2, B4 or B5, then we
put $D=(L_X-z)\cup \{w\}\subseteq I$.
Obviously, $D$ dominates $L$ and $|D|=|L_X|$.
For a component $L'$ of type B3, we put $D=L'_X\subseteq I$.
Although $D$ does not dominate $L'$, the set $I$ will dominate $L'$, since
the above vertex $w\in I$ dominates $L'_U$, and $L'_X$ dominates $L'_X\cup L'_{PN}$.
Let $L$ have type B6.
Put $D=L_Z\cup \{y\}\subseteq I$.
By Proposition 1, $y$ dominates $PN(x)$ for some $x\in L_X$.
The only possibility is that $y$ dominates $PN(f)$.
Moreover, $y$ dominates $L_U\cup\{f\}$.
Thus, $D$ dominates $L$ and $|D|=|L_X|$.
If $L'$ has type D1, then
put $D=\{f_1,f_2\}\subseteq I$, thus $|D|=|L'_X|=2$.
Although $D$ does not dominate $L'$, the set $I$ will dominate $L'$, since
the above $w\in I$ dominates $L'_U$, and $\{f_1,f_2\}$ dominates $L'-L'_U$.
If $L'$ is of type D2, then
put $D=\{f_1,y_1\}\subseteq I$
(or $D=\{f_2,y_2\}\subseteq I$ if $w\ad L'_{PN}\cup N'_1\cup \{f_2\}$, $w\nad \{u_2,f_1\}$),
thus $|D|=|L'_X|=2$.
Although $D$ does not dominate $L'$, the set $I$ will dominate $L'$, since
the above $w\in I$ and $D$ dominate $L'$.

{\bf Types $D'$, $D'_1$, $D'_2$; $D''$, $D''_1$, $D''_2$.}
For the component $L$ of type D we define two auxiliary subtypes $D'$ and $D''$.
We say that $L$ has {\it type} $D'$ if there is no vertex $w_0\in W$ such that
$w_0\ad \{f_1,f_2\}$ and $w_0\nad L_U\cup \{y_1,y_2\}$,
and $L$ has {\it type} $D''$ if such a vertex does exist.
Let $L$ have type $D'$.
We know that $\{y_1,y_2\}$ dominates $L$.
If neither $\{u_1,f_2\}$ nor $\{u_2,f_1\}$ dominate $L$, then
$L$ is called a {\it component of type} $D'_1$.
If at least one of these sets dominates $L$, say $\{u_1,f_2\}$
dominates $L$, then $L$ is of {\it type} $D'_2$.
Now, let $L$ have type $D''$.
By Lemma \ref{l6}, $y_1\ad y_2$, there are all edges between $N_1$ and
$N_2$, and w.l.o.g. (i) holds. 
Hence $\{u_1,f_2\}$ dominates $L$.
If $\{u_2,f_1\}$ does not dominate $L$, then $L$ has {\it type} $D''_1$,
and $L$ is of {\it type} $D''_2$ otherwise.
Note that $w_0$ is not adjacent to a vertex of other component, for
otherwise $\lz u_1,u_2,y_2,f_2,w_0,v\rz\cong P_6$ where $v\in N(w_0)-V(L)$.

{\bf Types D3--D24} of components are defined in the following procedure, where
we consider components from the set ${\cal D}$ determined by the procedure above.

\be
\item
If there are a component $L$ of type $D''$, a component $L'\in{\cal D}$,
and a vertex $w\in W$ such that
$w\nad \{y_1,f_1\}$ and $w$ dominates $N_2\cup \{f_2\}\cup L'$,
then $L$ is called a {\it component of type D3}
and $L'$ is called a {\it component of type D4}.
We put ${\cal D}={\cal D}-\{L,L'\}$ and start this step again.
If the above components do not exist, then go to the next step.

\item
If there are a component $L$ of type $D'$, a component $L'\in{\cal D}$,
and a vertex $w\in W$ such that
$w\nad \{f_1,f_2\}$ and $w$ dominates $N_2\cup \{y_1\}\cup L'$,
then $L$ has {\it type D5}
and $L'$ has {\it type D6}.
We put ${\cal D}={\cal D}-\{L,L'\}$ and start this step again.
If the above components do not exist, then go to the next step.

\item
If there are an $L$ of type $D''_1$, an $L'$ of type $D'_1$,
a $w\in W$ such that $w$ dominates
$L_{PN}\cup N_2\cup N_{1,2}\cup \{f_1\}\cup L'_U\cup \{f'_1,f'_2\}$,
$w\nad \{u_1,f_2,y'_1,y'_2\}$, and $N'_{1,2}\not=\es$, $y'_1\nad y'_2$,
then $L$ has {\it type D7}
and $L'$ has {\it type D8}.
We put ${\cal D}={\cal D}-\{L,L'\}$ and start this step again.
If the above components do not exist, then go to the next step.

\item
If there are an $L$ of type $D''_1$, an $L'$ of
type $D'_2$, and vertices $w,w'\in W$ such that
$\{w,w'\}\ad L_{PN}\cup \{f_1\}\cup N_2\cup N_{1,2}$,
$\{w,w'\}\nad \{u_1,f_2\}$,
and $w\nad\{y'_1,y'_2\}$, $w'\nad\{u'_1,f'_2\}$,
$w\ad L'_U\cup\{f'_1,f'_2\}$,
$w'\ad L'_{PN}\cup \{f'_1\}$,
then $L$ has {\it type D9}
and $L'$ has {\it type D10}.
We put ${\cal D}={\cal D}-\{L,L'\}$ and start this step again.
If the above components do not exist, then go to the next step.

\item
If there are an $L$ of type $D''_1$,
an $L'$ of type $D''_2$, and $w,w'\in W$ such that
$\{w,w'\}\ad L_{PN}\cup \{f_1\}\cup N_2\cup N_{1,2}$,
$\{w,w'\}\nad \{u_1,f_2\}$,
and $w\nad \{u'_1,f'_2\}$, $w'\nad \{u'_2,f'_1\}$,
\mbox{$w\ad L'_{PN}\cup \{f'_1\}\cup N'_2\cup N'_{1,2}$,}
$w'\ad L'_{PN}\cup \{f'_2\}\cup N'_1\cup N'_{1,2}$,
then $L$ has {\it type D11}
and $L'$ has {\it type D12}.
We put ${\cal D}={\cal D}-\{L,L'\}$ and start this step again.
If the above components do not exist, then go to the next step.

\item
If there are an $L$ of type $D''_2$,
an $L'$ of type $D'_1$, and $w,w'\in W$ such that
$w\nad \{u_1,f_2\}$, $w'\nad \{u_2,f_1\}$,
$w\ad L_{PN}\cup \{f_1\}\cup N_2\cup N_{1,2}$,
$w'\ad L_{PN}\cup \{f_2\}\cup N_1\cup N_{1,2}$,
and $\{w,w'\}\ad L'_U\cup\{f'_1,f'_2\}$, $\{w,w'\}\nad\{y'_1,y'_2\}$,
then $L$ has {\it type D13}
and $L'$ has {\it type D14}.
We put ${\cal D}={\cal D}-\{L,L'\}$ and start this step again.
If the above components do not exist, then go to the next step.

\item
If there are an $L$ of type $D''_2$,
an $L'$ of type $D'_2$, and $w,w'\in W$ such that
$w\nad \{u_1,f_2\}$, $w'\nad \{u_2,f_1\}$,
$w\ad L_{PN}\cup \{f_1\}\cup N_2\cup N_{1,2}$,
$w'\ad L_{PN}\cup \{f_2\}\cup N_1\cup N_{1,2}$,
and $w\nad\{y'_1,y'_2\}$, $w'\nad\{u'_1,f'_2\}$,
$w\ad L'_U\cup\{f'_1,f'_2\}$,
$w'\ad L'_{PN}\cup \{f'_1\}$,
then $L$ has {\it type D15}
and $L'$ has {\it type D16}.
We put ${\cal D}={\cal D}-\{L,L'\}$ and start this step again.
If the above components do not exist, then go to the next step.

\item
If there are an $L$ of type $D''_2$, an
$L'$ of type $D''_2$, and $w,w'\in W$ such that
$w\nad \{u_1,f_2\}$, $w'\nad \{u_2,f_1\}$,
$w\ad L_{PN}\cup \{f_1\}\cup N_2\cup N_{1,2}$,
$w'\ad L_{PN}\cup \{f_2\}\cup N_1\cup N_{1,2}$,
and $w\nad \{u'_1,f'_2\}$, $w'\nad \{u'_2,f'_1\}$,
$w\ad L'_{PN}\cup \{f'_1\}\cup N'_2\cup N'_{1,2}$,
$w'\ad L'_{PN}\cup \{f'_2\}\cup N'_1\cup N'_{1,2}$,
then $L$ has {\it type D17}
and $L'$ has {\it type D18}.
We put ${\cal D}={\cal D}-\{L,L'\}$ and start this step again.
If the above components do not exist, then go to the next step.

\item
If there are an $L$ of type $D'_1$, an
$L'$ of type $D'_2$, and $w,w'\in W$ such that
$\{w,w'\}\nad \{y_1,y_2\}$,
$\{w,w'\}\ad L_U\cup \{f_1,f_2\}$,
and $w\nad\{y'_1,y'_2\}$, $w'\nad\{u'_1,f'_2\}$,
$w\ad L'_U\cup\{f'_1,f'_2\}$, $w'\ad L'_{PN}\cup \{f'_1\}$,
then $L$ has {\it type D19}
and $L'$ has {\it type D20}.
We put ${\cal D}={\cal D}-\{L,L'\}$ and start this step again.
If the above components do not exist, then go to the next step.

\item
If there are an $L$ of type $D'_2$, an 
$L'$ of type $D'_2$, and $w,w'\in W$ such that
$w\nad \{u_1,f_2\}$, $w'\nad \{y_1,y_2\}$,
$w'\ad L_U\cup \{f_1,f_2\}$, $w\ad L_{PN}\cup \{f_1\}$,
and $w\nad\{y'_1,y'_2\}$, $w'\nad\{u'_1,f'_2\}$,
$w\ad L'_U\cup\{f'_1,f'_2\}$, $w'\ad L'_{PN}\cup \{f'_1\}$,
then $L$ has {\it type D21}
and $L'$ has {\it type D22}.
We put ${\cal D}={\cal D}-\{L,L'\}$ and start this step again.
If the above components do not exist, then go to the next step.

\item
If there are an $L$ of type $D'_2$, an 
$L'$ of type $D'_2$, and $w,w'\in W$ such that
$w\nad \{u_1,f_2\}$, $w'\nad \{y_1,y_2\}$,
$w'\ad L_U\cup \{f_1,f_2\}$, $w\ad L_{PN}\cup \{f_1\}$,
and $w'\nad\{y'_1,y'_2\}$, $w\nad\{u'_1,f'_2\}$,
$w'\ad L'_U\cup\{f'_1,f'_2\}$, $w\ad L'_{PN}\cup \{f'_1\}$,
then $L$ has {\it type D23}
and $L'$ has {\it type D24}.
We put ${\cal D}={\cal D}-\{L,L'\}$ and start this step again.
If the above components do not exist, then exit the procedure.
\ee

If $L$ has type D3 (D5), then there is an $L'$ of type D4 (D6). Put
\mbox{$D=\{y_1,f_1,f_2,w\}\subseteq I$.} It is obvious that $D$ dominates $L\cup
L'$. Let $L$ be of type D7 and $L'$ be of type D8. Put
$D=\{y_1,f_1,w,u\}\subseteq I$ where $u\in N'_{1,2}$. 
Note that 
$D$ dominates $(L\cup L')-L'_{PN}$. Let there be an $y\in
L'_{PN}$, say $y\in PN(f'_2)$, such that $y\nad D$. Since
$y'_1\nad y'_2$ and $y'_2\ad u'_2$, we conclude by Proposition 1
that $y'_2$ dominates $PN(f'_2)$. Thus $\lz
f_2,f_1,w,u,y'_2,y\rz\cong P_6$ and hence $D$ dominates $L\cup
L'$. If $L$ has type D9 (D11), then there is an $L'$ of type D10
(D12). Put $D=\{y_1,f_1,w,w'\}\subseteq I$. Obviously, $D$
dominates $L\cup L'$. If $L$ has type D13 (D15, D17), then there
is $L'$ of type D14 (D16, D18). Put $D=\{w,w',f'_1,f'_2\}\subseteq
I$. Obviously, $D$ dominates $L\cup L'$. If $L$ has type D19 (D21,
D23), then there is an $L'$ of type D20 (D22, D24). Put
$D=\{f_1,f_2,w,w'\}\subseteq I$. Obviously, $D$ dominates $L\cup
L'$.

{\bf Type D25.}
The set ${\cal D}$ determined by the above procedure
contains components of {\it type D25}.
We put
$$
W_1 = \{w\in W: N(w)\cap X\subseteq \cup_{L\in{\cal D}}L_X\}.
$$
A {\it $\Delta$-set} $S$ is formed by the following rule. For each
component $L\in {\cal D}$ we add in $S$ one of the 2-vertex
dominating sets. Thus, we add $\{y_1,y_2\}$ if $L$ is of type
$D'_1$, and $\{y_1,y_2\}$ or $\{u_1,f_2\}$ if $L$ has type $D'_2$.
Also, we add $\{u_1,f_2\}$ if $L$ is of type $D''_1$, and
$\{u_1,f_2\}$ or $\{u_2,f_1\}$ if $L$ has type $D''_2$. By the
next lemma, there is a $\Delta$-set $S$ dominating
$W_1\cup_{L\in{\cal D}}L$. We put $S\subseteq I$. Obviously,
$|S|=|\cup_{L\in{\cal D}}L_X|$.

\begin{lem}
\label{delta}
There exists a $\Delta$-set $S$ that dominates  $W_1\cup\left(\bigcup_{L\in{\cal D}}L\right)$.
\end{lem}

\pf
By definition, any $\Delta$-set $S$ dominates all components in ${\cal D}$.
Put $W(S)=W_1-N(S)$ and suppose that $W(S)\not=\es$ for any $\Delta$-set $S$.
Also, put $W_2=\cup W(S)$, where the union is taken over all $\Delta$-sets.
For $w\in W_2$, let ${\cal D}_w$ denote the set of components $L\in {\cal D}$ having
$N(w)\cap V(L)\not=\es$.

\begin{claim}
$|{\cal D}_w|\ge 2$ for any $w\in W_2$.
\end{claim}

\pf For any $w\in W_2$, $|N(w)\cap X|\ge 2$.
Assume that $N(w)\cap X$ contains vertices of one component $L\in {\cal D}$
only, i.e., $w\ad \{f_1,f_2\}$.
Since $w\in W(S)$ for some $\Delta$-set $S$,
it follows that the set $S$ must contain
$\{y_1,y_2\}\subset V(L)$ and the component $L$ is of type $D'$.
By the definition of $D'$, $w\ad u\in L_U$.
Hence, by Proposition 1, $w$ dominates $PN(x)$, where
$x\not\in\{f_1,f_2\}$, and by Lemma \ref{l1}, $x\ad x'\in X$.
Letting $y\in PN(x)$, we obtain $w\ad x$ or $w\ad x'$, for otherwise
$\lz y_1,f_1,w,y,x,x'\rz\cong P_6$.
\qed
\vspace{-0.3cm}

\begin{claim}
Let $L\in {\cal D}_w$ and $w\in W(S)$.
\be
\vspace*{-0.25cm}
\item[(1)] If $L$ has type $D'_1$ or $D'_2$ and $S\cap V(L)=\{y_1,y_2\}$, then
$w$ dominates $L_U\cup\{f_1,f_2\}$, $y_1\nad y_2$, and $N_1\ad N_2$.
\vspace*{-0.25cm}
\item[(2)] If $L$ has type $D'_2$ and $S\cap V(L)=\{u_1,f_2\}$, then $w$ dominates
$L_{PN}\cup \{f_1\}$.
\vspace*{-0.25cm}
\item[(3)] If $L$ is of type $D''_1$ or $D''_2$ and $S\cap V(L)=\{u_1,f_2\}$,
then $w$ dominates $L_{PN}\cup \{f_1,w_0\}\cup N_2\cup N_{1,2}$.
\vspace*{-0.25cm}
\item[(4)] If $L$ is of type $D''_2$ and $S\cap V(L)=\{u_2,f_1\}$,
then $w$ dominates $L_{PN}\cup \{f_2,w_0\}\cup N_1\cup N_{1,2}$.
\ee
\end{claim}

\pf By Claim 1, there is an $M\in {\cal D}_w$. Let us take
vertices $a,b\in V(M)$ such that $\lz a,b,w\rz\cong P_3$. $\;$
Such vertices do exist, since $N(w)\cap V(M)\not=\es$ and $w\nad
S\cap V(M)$. 

(1) Suppose that $w\nad \{f_1,f_2\}$. Since $L\in
{\cal D}_w$, $w$ is adjacent to a vertex from $L_U\cup L_{PN}$. If
$w\ad u\in L_U$, then w.l.o.g. $u\ad y_1$ and $\lz
a,b,w,u,y_1,f_1\rz\cong P_6$. If $w\ad y\in PN(f_i)$, then $\lz
a,b,w,y,f_1,f_2\rz\cong P_6$. Thus, $N(w)\cap L_F\not=\es$,
w.l.o.g. $w\ad f_1$. We have $w\ad f_2$, for otherwise $\lz
a,b,w,f_1,f_2,y_2\rz\cong P_6$. Now assume that $w\nad u\in L_U$.
W.l.o.g., let $u\ad y_1$. We have $\lz a,b,w,f_1,y_1,u\rz\cong
P_6$, a contradiction. Hence $w$ dominates $L_U\cup \{f_1,f_2\}$.
Since $\lz a,b,w,f_1,y_1,y_2\rz\not\cong P_6$, we obtain $y_1\nad
y_2$. If $u\in N_1$, $v\in N_2$ and $u\nad v$, then $\lz
u,y_1,f_1,f_2,y_2,v\rz\cong P_6$. Hence $N_1\ad N_2$. 

(2) Suppose
that $w\nad f_1$. If $w\ad y\in L_{PN}$, then $\lz
a,b,w,y,f_1,f_2\rz\cong P_6$. Hence $w\nad L_{PN}$ and $w\ad u\in
L_U$. By Proposition 1, $u\ad y_i$ for some $i\in\{1,2\}$. We have
$\lz a,b,w,u,y_i,f_i\rz\cong P_6$, a contradiction. Thus, $w\ad
f_1$. Let $y\in PN(f_1)$. Since $u_1$ dominates $PN(f_1)$, we have
$\lz a,b,w,f_1,y,u_1\rz\cong P_6$ if $w\nad y$. Hence  $w$
dominates $PN(f_1)$. Finally, if $y\in PN(f_2)$ and $y\nad w$,
then $\lz a,b,w,f_1,f_2,y\rz\cong P_6$. Thus, $w$ dominates
$PN(f_2)$. 

(3) Analogously to (2), we can show that $w$ dominates
$L_{PN}\cup \{f_1\}$. We know that $w_0\nad \{a,b\}$. Hence $w\ad
w_0$, for otherwise $\lz a,b,w,y_2,f_2,w_0\rz\cong P_6$. Now
assume that $w\nad u$ where $u\in N_2\cup N_{1,2}$. We have
$u_1\ad u$, since $u_1$ dominates $\hat N$. Also, $u\ad y_2$,
since $u\in N_2\cup N_{1,2}$. Now $\lz a,b,w,y_2,u,u_1\rz\cong
P_6$. Hence $w$ dominates $N_2\cup N_{1,2}$. 

(4) This case is
symmetric to (3). \qed

Let $x,y$ be arbitrary vertices from $W_2$.
We will show that either ${\cal D}_x\cap {\cal D}_y=\es$ or one of these sets
is a subset of other.
Suppose to the contrary that $L,M\in {\cal D}_x$, $N\not\in {\cal D}_x$
and $M,N\in {\cal D}_y$, $L\not\in {\cal D}_y$.
Obviously, $\lz V(L)\cup \{x\}\rz$ contains $P_3=\lz a,b,x\rz$,
$\lz V(N)\cup \{y\}\rz$ contains $P_3=\lz y,c,d\rz$.
Let $P$ be a shortest $(x-y)$-path in $\lz V(M)\cup \{x,y\}\rz$.
It is easy to see that the union of the above paths contains an induced $P_6$,
a contradiction.
Hence either ${\cal D}_x\cap {\cal D}_y=\es$, or
${\cal D}_x\subseteq {\cal D}_y$, or
${\cal D}_x\supseteq {\cal D}_y$.
Thus, there exist pairwise distinct maximal (with respect to inclusion)
sets ${\cal D}_{w_1},...,{\cal D}_{w_m}$ ($w_i\in W_2$)
such that
${\cal D}_w\subseteq {\cal D}_{w_i}$
for any $w\in W_2$ and some $i\in \{1,...,m\}$.

\begin{claim}
For any $w\in W_2$, ${\cal D}_w$ contains a component of type $D'_2$ or $D''_2$.
\end{claim}

\pf Assume to the contrary that ${\cal D}_w$ has components of types $D'_1$
and $D''_1$ only.
By Claim 1, $|{\cal D}_w|\ge 2$.
Firstly, suppose that ${\cal D}_w$ contains a component $L$ of type $D''_1$,
thus $w\nad \{u_1,f_2\}$.
By definition, $\{u_2,f_1\}$ does not dominate $L$, while
$f_1\ad PN(f_1)\cup \{f_2\}$ and $u_2\ad PN(f_2)\cup N_1$.
Therefore there is a $u\in N_{1,2}\cup N_2$ such that $u\ad \{u_1,y_2\}$,
$u\nad u_2$, and the edge $uy_1$ is undetermined.
If ${\cal D}_w$ contains another component $L'$ of type $D''_1$, then
$w\nad \{u'_1,f'_2\}$,
$\{u'_2,f'_1\}$ does not dominates $L'$ and we can show as above that
there is a vertex $u'\in L'_U$ with $u'\ad \{u'_1,y'_2\}$, $u'\nad u'_2$,
and the edge $u'y'_1$ is undetermined.
Applying Claim 2 (3) to $L$, $L'$ and $w$, we see that
$\lz w_0,f_1,f_2,y_1,y_2,u,u_1,u_2,w,w'_0,f'_1,f'_2,y'_1,y'_2,u',u'_1,u'_2
\rz\cong F_7$, $F_8$, or $F_9$, a contradiction.
Now let $L'$ be of type $D'_1$.
By Claim 2 (1,3), $w$ dominates $L_{PN}\cup \{f_1,w_0\}\cup N_2\cup N_{1,2}
\cup L'_U\cup \{f'_1,f'_2\}$, $w\nad \{u_1,f_2,y'_1,y'_2\}$, and $y'_1\nad y'_2$.
If $N'_{1,2}\not=\es$, then $L$ has type D7 and $L'$ is of type D8, a contradiction.
Hence $N'_{1,2}=\es$.
By definition, neither $\{u'_1,f'_2\}$ nor $\{u'_2,f'_1\}$  dominate $L'$.
Moreover, by Claim 2 (1), $N_1\ad N_2$.
The only possibility is that $u'_1$ does not dominate $N'_1$ and $u'_2$
does not dominate $N'_2$.
Thus, there is a $u'_3\in N'_1$ not adjacent to $u'_1$ and there is
an $u'_4\in N'_2$ not adjacent to $u'_2$.
Now
$\lz w_0,f_1,f_2,y_1,y_2,u,u_1,u_2,w,f'_1,f'_2,y'_1,y'_2,u'_1,u'_2,u'_3,u'_4
\rz\cong F_{10}$ or $F_{11}$, a contradiction.
Thus, all components of ${\cal D}_w$ have type $D'_1$.
Let $L\in{\cal D}_w$.
By Claim 2 (1), $w\ad L_U$, and by Proposition 1, $w$ dominates $PN(x)$.
Since $w\nad \{y_1,y_2\}$, we obtain $x\not\in L_X$ and $x\in L'_X$.
By Lemma \ref{l1}, $x\ad x'\in L'_X$.
Now $\lz y_1,f_1,w,y,x,x'\rz\not\cong P_6$ where $y\in PN(x)$.
Hence $w\ad x$ or $w\ad x'$, i.e., $L'\in {\cal D}_w$ and $L'$ has type $D'_1$.
This is a contradiction, since
$w$ dominates $PN(x)$ for $x\in \{f'_1,f'_2\}$ and
$w\nad \{y'_1,y'_2\}\subset S.$ \qed

The sets ${\cal D}_{w_i}$ and ${\cal D}_{w_j}$ have empty intersection for any
$i\not=j$. Hence, it is not difficult to construct a $\Delta$-set $S$
for which $\{w_1,...,w_m\}\subseteq W(S)$.
Now we form a $\Delta$-set $S'$ from $S$ as follows.
Firstly, put $S'=S$.
Further, for any $i\in \{1,...,m\}$ and for any component $L\in {\cal D}_{w_i}$ of types
$D'_2$ and $D''_2$ we replace a 2-vertex dominating set $S'\cap V(L)$ by another
2-vertex dominating set.
More precisely, if $L$ has type $D'_2$ and $\{y_1,y_2\}\in S$, then
we delete this set from $S'$ and add $\{u_1,f_2\}$ in $S'$, and
vice versa. In the same way we exchange the dominating sets $\{u_1,f_2\}$
and $\{u_2,f_1\}$ for components of type $D''_2$.
Since $W(S)\not=\es$ for any $\Delta$-set $S$, there is a $w'\in W(S')$.
By Claim 3, ${\cal D}_{w_i}$ contains a component of type $D'_2$ or $D''_2$ for
any $i\in \{1,...,m\}$.
It follows from Claim 2 and from the definition of $S'$ that
$w_i$ is adjacent to a vertex of $S'$ for any $i\in \{1,...,m\}$.
Thus, $w'\not\in\{w_1,...,w_m\}$.
Since ${\cal D}_{w_i}$ are maximal sets, we conclude that
for some $w\in\{w_1,...,w_m\}$,
$$
{\cal D}_{w'}\subseteq {\cal D}_w.
$$

Firstly, suppose that ${\cal D}_{w'}$ contains a component $L$ of type $D''_1$.
By Claim 2 (3), both $w$ and $w'$ dominate
$L_{PN}\cup \{f_1\}\cup N_2\cup N_{1,2}$ and $\{w,w'\}\nad \{u_1,f_2\}$.
By Claim 3, ${\cal D}_{w'}$ has a component $L'$ of type $D'_2$ or $D''_2$.
Let $L'$ be of type $D'_2$.
W.l.o.g., let $w\nad\{y'_1,y'_2\}$ and $w'\nad\{u'_1,f'_2\}$.
By Claim 2 (1), $w$ dominates $L'_U\cup\{f'_1,f'_2\}$, and by Claim 2 (2),
$w'$ dominates $L'_{PN}\cup \{f'_1\}$.
Therefore $L$ has type D9 and $L'$ has type D10, contrary to the
definition of ${\cal D}$.
Now let $L'$ be of type $D''_2$.
W.l.o.g., let $w\nad \{u'_1,f'_2\}$ and $w'\nad \{u'_2,f'_1\}$.
By Claim 2 (3), $w$ dominates $L'_{PN}\cup \{f'_1\}\cup N'_2\cup N'_{1,2}$.
By Claim 2 (4), $w'$ dominates $L'_{PN}\cup \{f'_2\}\cup N'_1\cup N'_{1,2}$.
Hence $L$ has type D11 and $L'$ has type D12, a contradiction.

Therefore, we may assume that neither $L$ nor $L'$ has type $D''_1$.
Suppose that $L$ is of type $D''_2$.
W.l.o.g., let $w\nad \{u_1,f_2\}$ and $w'\nad \{u_2,f_1\}$.
By Claim 2 (3,4), $w$ dominates $L_{PN}\cup \{f_1\}\cup N_2\cup N_{1,2}$
and $w'$ dominates $L_{PN}\cup \{f_2\}\cup N_1\cup N_{1,2}$.
There are three cases to consider.
Let $L'$ be of type $D'_1$.
By Claim 2 (1), $\{w,w'\}\ad L'_U\cup\{f'_1,f'_2\}$ and $\{w,w'\}\nad\{y'_1,y'_2\}$.
Hence $L$ has type D13 and $L'$ has type D14, a contradiction.
Let $L'$ be of type $D'_2$.
W.l.o.g., let $w\nad\{y'_1,y'_2\}$ and $w'\nad\{u'_1,f'_2\}$.
By Claim 2 (1,2), $w$ dominates $L'_U\cup\{f'_1,f'_2\}$ and
$w'$ dominates $L'_{PN}\cup \{f'_1\}$.
Therefore $L$ has type D15 and $L'$ has type D16, a contradiction.
Finally, let $L'$ be of type $D''_2$.
W.l.o.g., let $w\nad \{u'_1,f'_2\}$ and $w'\nad \{u'_2,f'_1\}$.
By Claim 2 (3), $w$ dominates $L'_{PN}\cup \{f'_1\}\cup N'_2\cup N'_{1,2}$.
By Claim 2 (4), $w'$ dominates $L'_{PN}\cup \{f'_2\}\cup N'_1\cup N'_{1,2}$.
Hence $L$ has type D17 and $L'$ has type D18, a contradiction.

Thus, ${\cal D}_{w'}$ contains components of types $D'_1$ and $D'_2$ only.
By Claim 3, ${\cal D}_{w'}$ has a component $L'$ of type $D'_2$.
There are two cases to consider.
Firstly, let $L$ be of type $D'_1$, and hence $\{w,w'\}\nad \{y_1,y_2\}$.
By Claim 2 (1), $\{w,w'\}\ad L_U\cup \{f_1,f_2\}$.
W.l.o.g., let $w\nad\{y'_1,y'_2\}$ and $w'\nad\{u'_1,f'_2\}$.
By Claim 2 $(1,2)$, $w$ dominates $L'_U\cup\{f'_1,f'_2\}$ and
$w'$ dominates $L'_{PN}\cup \{f'_1\}$.
Therefore, $L$ has type D19 and $L'$ has type D20, a contradiction.
Finally, let both $L$ and $L'$ be of type $D'_2$.
Let $w\nad \{u_1,f_2\}$ and $w'\nad \{y_1,y_2\}$.
By Claim 2 (1), $w'$ dominates $L_U\cup \{f_1,f_2\}$, and
by Claim 2 (2), $w$ dominates $L_{PN}\cup \{f_1\}$.
We consider two subcases.
Let $w\nad\{y'_1,y'_2\}$ and $w'\nad\{u'_1,f'_2\}$.
By Claim 2 $(1,2)$, $w$ dominates $L'_U\cup\{f'_1,f'_2\}$ and
$w'$ dominates $L'_{PN}\cup \{f'_1\}$.
Hence $L$ has type D21 and $L'$ has type D22, a contradiction.
Finally, let $w'\nad\{y'_1,y'_2\}$ and $w\nad\{u'_1,f'_2\}$.
By Claim 2 $(1,2)$, $w'$ dominates $L'_U\cup\{f'_1,f'_2\}$ and
$w$ dominates $L'_{PN}\cup \{f'_1\}$.
Hence $L$ has type D23 and $L'$ has type D24, a contradiction.
The proof of Lemma \ref{delta} is complete. \qed

Now we subdivide the types F and G into subtypes.
The sets $N_1, N_2, N_{1,2}$ used here are defined by (1).

{\bf Type F1.} $L_Z=\{z\}$, $y_1$ dominates $PN(z)$, and there is
no vertex $w\in W$ dominating $L_U\cup PN(z)$.
Also, there is no $u\in N_{1,2}$ with $u\ad u_1$.
Put $D=\{y_1,y_2,f_2\}\subseteq I$.
Obviously, $D$ dominates $L$ and $|D|=3=|L_X|$.

{\bf Type F2.} $L_Z=\{z\}$, $y_1$ dominates $PN(z)$, and there is
no vertex $w\in W$ dominating $L_U\cup PN(z)$.
Also, there is a $u\in N_{1,2}$ with $u\ad u_1$.
Put $D=\{y,y_1,y_2\}\subseteq I$ where $y\in PN(z)$.
Observe that $D$ dominates $L-PN(f_2)$.
Each vertex $g\in PN(f_2)$ is adjacent to $u_2$ and,
by Proposition 1, $g$ dominates $PN(x)$ for $x\in L_X$.
Hence $PN(f_2)$ is dominated by $D$.
Thus, $D$ dominates $L$ and $|D|=3=|L_X|$.

{\bf Type F3.} $L_Z=\{z\}$, $y_1$ dominates $PN(z)$, and there is
a $w\in W$ dominating $L_U\cup PN(z)$.
Put $D=\{f_1,f_2,w\}\subseteq I$.
Clearly, $D$ dominates $L$ and $|D|=3=|L_X|$.

{\bf Type F4.} $L_Z=\{z\}$ and $y_1$ does not dominate $PN(z)$,
i.e., there is a $y\in PN(z)$ with $y\nad y_1$. Since
$\{u_1,y_1\}$ dominates $PN(z)$, we obtain $u_1\ad y$. Put
$D=\{y,y_1,y_2\}\subseteq I$. We know that $D$ dominates $L_U\cup
PN(f_1)\cup L_X$. Any vertex $g\in PN(f_2)$ is adjacent to $u_2$
and, by Proposition 1, $g$ dominates $PN(x)$ for some $x\in L_X$.
Therefore, $D$ dominates $PN(f_2)$. Suppose now that $D$ does not
dominate $PN(z)$, i.e., there is a $g\in PN(z)$ and $g\nad y_1$.
We have $g\ad u_1$, since $\{u_1,y_1\}$ dominates $PN(z)$. By
Proposition 1, $g$ dominates $PN(x)$ for $x\in L_X$, i.e., $g$ is
adjacent to $y,y_1$ or $y_2$, a contradiction. Thus, $D$ dominates
$L$ and $|D|=3=|L_X|$.

{\bf Type F5.} $|L_Z|\ge 2$.
Put $D=\{u_1,y_1,y_2,f_2\}\subseteq I$.
We know that $\{u_1,y_1\}$ dominates $PN(z)$ for any $z\in L_Z$.
Also, $\{y_1,y_2\}$ dominates $L_U\cup PN(f_1)$
and $f_2$ dominates \mbox{$PN(f_2)\cup L_X$.}
Thus, $D$ dominates $L$ and $|D|=4\le |L_X|$.

{\bf Type G1.} The set $\{y_1,y_2\}$ does not dominate $PN(f_3)$.
Let $y_3\in PN(f_3)$ and $y_3\nad \{y_1,y_2\}$.
Put $D=\{y_1,y_2,y_3\}\subseteq  I$.
Obviously, $D$ dominates $L_U\cup L_X$.
Any vertex $g\in L_{PN}$ is adjacent to $u\in L_U$.
By Proposition 1, $g$ dominates $PN(x)$ for some $x\in L_X$,
i.e., $g$ is dominated by $D$.
Thus, $D$ dominates $L$ and $|D|=|L_X|$.

{\bf Type G2.} The set $\{y_1,y_2\}$ dominates $PN(f_3)$
but it does not dominate $L_U$.
Put $D=\{y_1,y_2,y_3\}\subseteq  I$, where $y_3\in PN(f_3)$.
By the same argument as for the type G1, $D$ dominates $L$ and $|D|=|L_X|$.

{\bf Type G3.} The set $\{y_1,y_2\}$ dominates $PN(f_3)\cup L_U$,
and there is no $w\in W$ dominating $PN(f_3)\cup L_U$.
Also, there is no $u\in N_{1,2}$ with $u\ad u_1$.
Put $D=\{y_1,y_2,f_2\}\subseteq  I$.
Obviously, $D$ dominates $L$ and $|D|=|L_X|$.

{\bf Type G4.} The set $\{y_1,y_2\}$ dominates $PN(f_3)\cup L_U$,
and there is no $w\in W$ dominating $PN(f_3)\cup L_U$.
Also, there is a $u\in N_{1,2}$ with $u\ad u_1$.
Put $D=\{y_1,y_2,y_3\}\subseteq  I$ where $y_3\in PN(f_3)$.
By the same argument as in the type G1,
$D$ dominates $L$ and $|D|=|L_X|$.

{\bf Type G5.} The set $\{y_1,y_2\}$ dominates $PN(f_3)\cup L_U$,
and there is a $w\in W$ dominating $PN(f_3)\cup L_U$.
Put $D=\{f_1,f_2,w\}\subseteq  I$.
Clearly, $D$ dominates $L$ and $|D|=|L_X|$.

\begin{lem}
\label{l8}
There exists a vertex $w^*\in W-W_1$ such that $w^*\nad I$,
the set $X^*=N(w^*)\cap X$ contains a vertex of a component of type
different from D25, and  $|X^*|\ge 2$.
\end{lem}

\pf It follows from the construction of the set $I$ that $|I|\le
|X|$, and $I$ dominates $(H-W)\cup W_1$. We obtain $|I|\le
|X|=ir(H)<\gamma(H)$. Consequently, $I$ does not dominate $H$.
Thus, there is a vertex $w^*\in W-W_1$ such that $w^*\nad I$.
Since $w^*\not\in W_1$, $w^*$ is adjacent to a vertex of a
component of type different from D25. Moreover, by the
definition of $W$, $|X^*|\ge 2$. \qed \vspace{-0.3cm}

\begin{lem}
\label{l9}
The set $X^*$ contains no vertex from a component of type D1--D24.
\end{lem}

\pf
Suppose to the contrary that $X^*$ contains a vertex of a component $M$
having one of the types D1--D24.
The further proof is based on the next claims.

\begin{claim}
If $X^*\cap M_X\not=\es$, 
then $M$ cannot be of type
D1--D9, D11, D14, D16, D18, D19, D21 or D23.
\end{claim}

\pf
Since $X^*\cap I=\es$, $M$ cannot be of types D1, D3, D5, D14, D16, D18,
D19, D21, D23. Next, we examine the remaining cases individually and derive a contradiction in each case.
\smallskip

\noindent
(1)
If $M$ is of type D2, then $w^*\ad f_2$ and there are the
corresponding component $L$ of type B5 and a vertex $w$ such that
$w\ad f\in V(L)$, $w\nad z\in V(L)$, $w\ad f_1\in V(M)$ and $w\nad
f_2\in V(M)$. We obtain $w^*\ad z$, for otherwise $\lz
w^*,f_2,f_1,w,f,z\rz\cong P_6$. Now $\lz
f,z,w^*,f_2,f_1,y_1\rz\cong P_6$, a contradiction.
\smallskip

\noindent
(2)
Let $M$ be of type D4 and $w^*\ad x\in M_X$.
By definition, there are a component $L$ of type D3 with $y_1,f_1,f_2\in V(L)$ and a
$w\in W$ such that $w^*\nad \{y_1,f_1,f_2,w\}$.
Now $\lz y_1,f_1,f_2,w,x,w^*\rz\cong P_6$, a contradiction.
We can show analogously that $M$ cannot be of type D6.
\smallskip

\noindent
(3)
Let $M$ be of type D7, $w^*\ad f_2$, and $M'$ be the corresponding
component of type D8 and $u'\in N'_{1,2}$. Since $\lz
w^*,f_2,f_1,w,u',y'_1\rz\not\cong P_6$, we obtain $w^*\ad y'_1$.
Now $\lz y_1,f_1,f_2,w^*,y'_1,u'\rz\cong P_6$, a contradiction.
\smallskip

\noindent
(4)
Let $M$ be of type D8, $u'\in N'_{1,2}$, and $L$ be the corresponding component
of type D7.
We have proved above that $w^*\nad f_2$.
Suppose w.l.o.g. that $w^*\ad f'_1$.
We have $w^*\nad y'_1$, for otherwise $\lz w^*,y'_1,u',w,f_1,f_2\rz\cong P_6$.
Analogously, $w^*\nad y'_2$.
By definition, $M$ has type $D'_1$ and hence $w^*\ad v'\in M_U$.
By Proposition 1, $w^*$ dominates $PN(x)$ for $x\not=f'_1,f'_2$, i.e.,
$w^*\ad g\in PN(x)$.
We have $y'_1\nad y'_2$, for otherwise $\lz f_2,f_1,w,f'_1,y'_1,y'_2\rz\cong P_6$.
Now $\lz y'_2,u',y'_1,f'_1,w^*,g\rz\cong P_6$, a contradiction.
\smallskip

\noindent
(5)
Let $M$ be of type D9, $w^*\ad f_2$, and $M'$ be the corresponding component
of type D10.
Since $\lz w^*,f_2,f_1,w',f'_1,f'_2\rz\not\cong P_6$, we conclude
that $w^*\ad f'_1$ or $w^*\ad f'_2$.
If $x\in V(M')$ is adjacent to $y\in V(M')$, and $w^*\ad x$, $w^*\nad y$,
then $\lz y_1,f_1,f_2,w^*,x,y\rz\not\cong P_6$ implies $w\ad y$.
Therefore, $w^*$ dominates $M'$.
Let $x\in N_2$.
If $w^*\nad x$, then $\lz f'_2,w^*,f_2,f_1,w',x\rz\cong P_6$.
Hence $w^*$ dominates $N_2$.
By definition, $M$ must be of type D3 and $M'$ must be of type D4, a contradiction,
since these types were defined before the types D9 and D10.
\smallskip

\noindent
(6)
Let $M$ be of type D11, $w^*\ad f_2$, and $M'$ be the corresponding component
of type D12.
Since $\lz w^*,f_2,f_1,w',f'_2,f'_1\rz\not\cong P_6$, we conclude
that $w^*\ad f'_1$ or $w^*\ad f'_2$.
If $x\in V(M')$ is adjacent to $y\in V(M')$, and $w^*\ad x$, $w^*\nad y$,
then $\lz y_1,f_1,f_2,w^*,x,y\rz\not\cong P_6$ implies $w\ad y$.
Therefore $w^*$ dominates $M'$.
Let $x\in N_2$.
If $w^*\nad x$, then $\lz f'_1,w^*,f_2,f_1,w',x\rz\cong P_6$.
Hence $w^*$ dominates $N_2$.
By definition, $M$ must be of type D3 and $M'$ must be of type D4, a contradiction.
\qed
\vspace{-0.3cm}

\begin{claim}
If $|X^*\cap M_X|=2$, then $M$ cannot be of type
D4, D6, D10, D12, D13, D15, D17, D20, D22 or D24.
\end{claim}

\pf
Suppose to the contrary that $X^*$ contains both vertices of $M_X$, where
$M$ has one of the above types.
By Claim 4, $M$ cannot be of type D4 or D6. As before, we will consider all other possibilities one by one, and derive a contradiction in each case.
\smallskip

\noindent
(1) Let $M$ be of type D10, $w^*\ad \{f'_1,f'_2\}$, and $L$ be the corresponding component
of type D9.
By Claim 4, $w^*\nad f_2\in L_X$.
We have $w^*\nad y'_2$, for otherwise $\lz y'_2,w^*,f'_1,w,f_1,f_2\rz\cong P_6$.
Now $\lz w^*,f'_2,y'_2,w',f_1,f_2\rz\cong P_6$, a contradiction.

\smallskip

\noindent
(2)
Let $M$ be of type D12, $w^*\ad \{f'_1,f'_2\}$, and $L$ be the corresponding component
of type D11.
By Claim 4, $w^*\nad f_2\in L_X$.
We have $w^*\nad u'_1$, for otherwise $\lz f'_1,w^*,u'_1,w',f_1,f_2\rz\cong P_6$.
Analogously, $w^*\nad u'_2$.
Also, $u'_1\ad u'_2$, since $M$ has type $D''_2$.
We have $w\ad w'$, for otherwise $\lz w^*,f'_1,w,u'_2,u'_1,w'\rz\cong P_6$.
Now $\lz w^*,f'_1,f'_2,w,w',y_1,u'_1,u'_2\rz\cong F_2$, a contradiction.

\smallskip

\noindent
(3)
Let $M$ be of type D13, $w^*\ad \{f_1,f_2\}$, and $M'$ be the
corresponding component of type D14. Since $M$ is of type $D''_2$,
we obtain $u_1\ad u_2$. Suppose that $w^*\nad \{u_1,u_2\}$. We
obtain $w\ad w'$, for otherwise $\lz w^*,f_1,w,u_2,u_1,w'\rz\cong
P_6$. Now $\lz w^*,f_1,f_2,w,w',f'_1,u_1,u_2\rz\cong F_2$.
Therefore, w.l.o.g. $w^*\ad u_1$. Now $w^*\ad y'_1$, for otherwise
$\lz f_1,w^*,u_1,w',f'_1,y'_1\rz\cong P_6$. Further, if $x\in
V(M)$ is adjacent to $y\in V(M)$ and $w^*\ad x$, $w^*\nad y$, then
$\lz f'_2,f'_1,y'_1,w^*,x,y\rz\cong P_6$. Therefore $w^*$
dominates $M$. 
Let $x\in N'_2$. If $w^*\nad x$, then $\lz
f_2,w^*,y'_1,f'_1,w,x\rz\cong P_6$. Hence $w^*$ dominates $N'_2$.
By definition, $M$ must be of type D6 and $M'$ must be of type D5,
a contradiction.

\smallskip

\noindent
(4)
Let $M$ be of type D15, $w^*\ad \{f_1,f_2\}$, and $M'$ be the
corresponding component of type D16. Since $M$ is of type $D''_2$,
we obtain $u_1\ad u_2$. Suppose that $w^*\nad \{u_1,u_2\}$. We
obtain $w\ad w'$, for otherwise $\lz w^*,f_1,w,u_2,u_1,w'\rz\cong
P_6$. Now $\lz w^*,f_1,f_2,w,w',f'_1,u_1,u_2\rz\cong F_2$. If
$w^*\ad u_1$, then $\lz f_1,w^*,u_1,w',f'_1,f'_2\rz\cong P_6$, and
if $w^*\ad u_2$, then $\lz u_2,w^*,f_2,w',f'_1,f'_2\rz\cong P_6$,
a contradiction.

\smallskip

\noindent
(5)
Let $M$ be of type D17, $w^*\ad \{f_1,f_2\}$, and $M'$ be the
corresponding component of type D18. Since $M$ is of type $D''_2$,
we obtain $u_1\ad u_2$. Suppose that $w^*\nad \{u_1,u_2\}$. We
obtain $w\ad w'$, for otherwise $\lz w^*,f_1,w,u_2,u_1,w'\rz\cong
P_6$. We have $\lz w^*,f_1,f_2,w,w',u_1,f'_1\rz\cong F_1$.
Therefore, w.l.o.g. $w^*\ad u_1$. Now $\lz
f_1,w^*,u_1,w',f'_2,f'_1\rz\cong P_6$, a contradiction.

\smallskip

\noindent
(6)
Let $M$ be of type D20, $w^*\ad \{f'_1,f'_2\}$, and $L$ be the corresponding component
of type D19.
Suppose $w^*\nad y'_2$.
We have $w^*\ad y_1$, for otherwise $\lz w^*,f'_2,y'_2,w',f_1,y_1\rz\cong P_6$.
Now $\lz f_2,f_1,y_1,w^*,f'_2,y'_2\rz\cong P_6$.
Therefore $w^*\ad y'_2$.
We have $w^*\ad y_1$, for otherwise $\lz y'_2,w^*,f'_1,w,f_1,y_1\rz\cong P_6$.
Let $x\in V(M)$ be adjacent to $y\in V(M)$ and $w^*\ad x$, $w^*\nad y$.
We obtain $w^*\ad y$, for otherwise $\lz f_2,f_1,y_1,w^*,x,y\rz\cong P_6$.
Therefore, $w^*$ dominates $M$.
Assume that $u\in N_2$ and $u\nad w^*$.
We have $\lz f'_2,w^*,y_1,f_1,w',u\rz\cong P_6$.
Hence $w^*$ dominates $N_2$.
Thus, by definition, $L$ must be of type D5 and $M$ must be of type D6,
a contradiction.

\smallskip

\noindent
(7)
Let $M$ be of type D22, $w^*\ad \{f'_1,f'_2\}$, and $L$ be the corresponding component
of type D21.
We have $w^*\nad y'_2$, for otherwise $\lz y'_2,w^*,f'_1,w,f_1,f_2\rz\cong P_6$.
Now $w^*\ad y_1$, for otherwise $\lz w^*,f'_2,y'_2,w',f_1,y_1\rz\cong P_6$.
Hence $\lz f_2,f_1,y_1,w^*,f'_2,y'_2\rz\cong P_6$, a contradiction.

\smallskip

\noindent
(8)
Finally, let $M$ be of type D24, $w^*\ad \{f'_1,f'_2\}$, and $L$
be the corresponding component of type D23. We have $w^*\ad y'_2$,
for otherwise $\lz w^*,f'_2,y'_2,w,f_1,f_2\rz\cong P_6$. Now
$w^*\ad y_1$, for otherwise $\lz y'_2,w^*,f'_1,w',f_1,y_1\rz\cong
P_6$. Since $\lz f_2,f_1,y_1,w^*,y'_2,u'_2\rz\not\cong P_6$, we
obtain $w^*\ad u'_2$. The set $\{u'_1,f'_2\}$ dominates $M$
because $M$ has type $D'_2$. Hence $u'_1\ad u'_2$. We have $w\ad
u'_2$, for otherwise $\lz f_2,f_1,w,y'_1,u'_1,u'_2\rz\cong P_6$.
Hence $\lz f_2,f_1,w,u'_2,w^*,f'_2\rz\cong P_6$, a contradiction.
\qed

By Claim 4, $M$ may have one of the types D10, D12, D13, D15, D17, D20, D22 or D24,
and, by Claim 5, $|X^*\cap M_X|=1$.

Let $M$ be of type D10 and $L$ be the corresponding component of type D9.
By Claim 4, $w^*\nad f_2$.
If $w^*\nad f'_1$, then $w^*\ad f'_2$ and $\lz w^*,f'_2,f'_1,w',f_1,f_2\rz\cong P_6$.
Hence $w^*\ad f'_1$ and $w^*\nad f'_2$.
We have $w^*\nad y'_2$, for otherwise $\lz y'_2,w^*,f'_1,w,f_1,f_2\rz\cong P_6$.
Since $|X^*|\ge 2$, we have $w^*\ad v\in R$ where $R$ is another component.
If $x\in V(R)$ is adjacent to $y\in V(R)$ and $w^*\ad x$, $w^*\nad y$,
then $\lz y'_2,f'_2,f'_1,w^*,x,y\rz\cong P_6$.
Therefore $w^*$ dominates $R$.
Since $w^*\nad I$, we see that $R$ may be of type D4, D6, D10, D12, D13, D15, D17, D20, D22 or D24.
This contradicts to Claim 5 because $|X^*\cap R_X|=2$.

Let $M$ be of type D12 and $L$ be the corresponding component of type D11.
By Claim 4, $w^*\nad f_2$.
If $w^*\ad f'_1$ then $\lz w^*,f'_1,f'_2,w',f_1,f_2\rz\cong P_6$.
If $w^*\ad f'_2$ then $\lz w^*,f'_2,f'_1,w,f_1,f_2\rz\cong P_6$,
a contradiction.

Let $M$ be of type D13, and $M'$ be the corresponding component
of type D14.
Suppose w.l.o.g. that $w^*\ad f_2$ and $w^*\nad f_1$.
Since $\lz w^*,f_2,f_1,w,f'_1,y'_1\rz\not\cong P_6$, we conclude
that $w^*\ad y'_1$.
Now $\lz f_1,f_2,w^*,y'_1,f'_1,f'_2\rz\cong P_6$, a contradiction.

Let $M$ be of type D15, and $M'$ be the corresponding component
of type D16.
If $w^*\ad f_1$ then $\lz w^*,f_1,f_2,w',f'_1,f'_2\rz\cong P_6$.
Suppose that $w^*\ad f_2$.
Since $\lz w^*,f_2,f_1,w,f'_1,y'_1\rz\not\cong P_6$, we conclude
that $w^*\ad y'_1$.
Now $\lz f_1,f_2,w^*,y'_1,f'_1,f'_2\rz\cong P_6$, a contradiction.

Let $M$ be of type D17, and $M'$ be the corresponding component
of type D18.
Suppose w.l.o.g. that $w^*\ad f_2$.
We have $\lz w^*,f_2,f_1,w,f'_1,f'_2\rz\cong P_6$, a contradiction.

Let $M$ be of type D20, and $L$ be the corresponding component
of type D19.
If $w^*\ad y_1$, then $\lz f_2,f_1,y_1,w^*,f'_1,f'_2\rz\cong P_6$,
since $|X^*\cap M_X|=1$.
Hence $w^*\nad y_1$.
Suppose that $w^*\ad f'_2$.
We have $\lz w^*,f'_2,f'_1,w',f_1,y_1\rz\cong P_6$.
Hence $w^*\ad f'_1$.
We obtain $w^*\nad y'_2$, for otherwise $\lz y'_2,w^*,f'_1,w,f_1,y_1\rz\cong P_6$.
Since $|X^*|\ge 2$, we have $w^*\ad v\in R$ where $R$ is another component.
If $x\in V(R)$ is adjacent to $y\in V(R)$ and $w^*\ad x$, $w^*\nad y$,
then $\lz y'_2,f'_2,f'_1,w^*,x,y\rz\cong P_6$.
Therefore $w^*$ dominates $R$.
Since $w^*\nad I$, we see that $R$ may be of type D4, D6, D10, D12, D13, D15, D17, D20, D22 or D24.
This contradicts to Claim 5 because $|X^*\cap R_X|=2$.
We can show analogously that $M$ cannot be of type D22.

Let $M$ be of type D24, and $L$ be the corresponding component
of type D23.
If $w^*\ad y_1$, then $\lz f_2,f_1,y_1,w^*,f'_1,f'_2\rz\cong P_6$.
Hence $w^*\nad y_1$.
Suppose that $w^*\ad f'_2$.
We have $\lz w^*,f'_2,f'_1,w,f_1,y_1\rz\cong P_6$.
Hence $w^*\ad f'_1$.
We obtain $w^*\nad y'_2$, for otherwise $\lz y'_2,w^*,f'_1,w',f_1,y_1\rz\cong P_6$.
Applying the same argument as for type D20 we easily deduce a contradiction.
\qed
\vspace{-0.3cm}

\begin{lem}
\label{l10}
The set $X^*$ contains vertices of at least two components of the
graph $H-W$.
\end{lem}

\pf
Suppose to the contrary that $X^*$ only contains vertices of
one component $M$.
By Lemma \ref{l9}, $M$ cannot be of types D1--D24.
By Lemma \ref{l8}, $M$ cannot have type D25.
Furthermore, the component $M$ cannot be of type
A1--A3, B1--B6, F3, G5 or H,
since for any of these types $|M_X-I|=0$ or 1, while $|X^*|\ge 2$.
If $M$ is of type E, then $N(w^*)\cap M_X=\{z_1,z_2\}$ and
$\lz z_1,w^*,z_2,f_2,y_2,y_1\rz\cong P_6$, a contradiction.
Thus, $M$ may have one of the types A4, F1, F2, F4, F5, or G1--G4.

Let $M$ have type A4 and $N(w^*)\cap X=\{f_2,z\}$, where $f_2,z\in
M_X$. Since $w^*\nad I$, we have $w^*\nad \{f_1,y_1,y_2\}$. If
there is a $u\in N_2$ such that $w^*\nad u$, then $\lz
u,y_2,y_1,f_1,z,w^*\rz\cong P_6$. Therefore, $w^*$ dominates
$N_2$. Let there exist an $u\in N_1\cup N_{1,2}$ such that
$w^*\nad u$. By Lemma \ref{l4}, $u\nad u_2\in N_2$. We obtain $\lz
u,y_1,f_1,z,w^*,u_2\rz\cong P_6$. Therefore, $w^*$ dominates
$N_1\cup N_{1,2}$. Thus, $w^*$ dominates $M_U$. Since $w^*\ad
u_1\in U$, it follows by Proposition 1 that $w^*$ dominates
$PN(x)$ for some $x\in X$. Clearly, $x\not=f_1,f_2$. Suppose that
$x\not\in M_Z$ and hence $x\in M'_X$, where $M'$ is another
component of $H-W$. Let $w^*\ad y$, where $y\in PN(x)$. We have
$\lz y_2,y_1,f_1,z,w^*,y\rz\cong P_6$. Therefore, $x\in M_Z$.
Thus, $w^*$ dominates $M_U\cup PN(x)$ where $x\in M_Z$ and
$w^*\nad \{f_1,y_1,y_2\}$, i.e., $M$ is a component of type A3, a
contradiction.

Now let $M$ have type F1 and $w^*\ad \{f_1,z\}$. Suppose that
there is a $u\in N_1\cup N_2$ such that $u\nad w^*$. If $u\in
N_1$, then $\lz u,y_1,y_2,f_2,z,w^*\rz\cong P_6$, and if $u\in
N_2$, then $\lz u,y_2,y_1,f_1,w^*,z\rz\cong P_6$. Hence $w^*$
dominates $N_1\cup N_2\not=\es$. 

By Proposition 1, $w^*$ dominates
$PN(x)$ for some $x\in X$. If $x\in M'_X$ for another component
$M'$ of $H-W$, then $w^*\ad g$ for $g\in PN(x)$ and $\lz
y_1,y_2,f_2,z,w^*,g\rz\cong P_6$. Since $w^*\nad \{y_1,y_2\}$,
$w^*$ dominates $PN(z)$. By the hypothesis, $w^*$ cannot dominate
$M_U\cup PN(z)$, and hence there is a $u\in N_{1,2}$ such that
$w^*\nad u$. By another hypothesis, $u\nad u_1$. We obtain $\lz
u,y_1,u_1,w^*,z,f_2\rz\cong P_6$. 

Let $M$ be of type F2, and
suppose that $w^*\ad f_1$. If $w^*\ad f_2$, then $w^*$ is adjacent
to $u_1$ or $u_2$, for otherwise $\lz
w^*,f_1,f_2,y_1,y_2,u_1,u_2\rz\cong F_1$. If $w^*\ad z$, then
$w^*\ad u_2$, for otherwise $\lz u_2,y_2,y_1,f_1,w^*,z\rz\cong
P_6$. Therefore, by Proposition 1, $w^*$ dominates $PN(x)$ for
$x\in M'_X$, where $M'$ is another component of $H-W$. We have
$w^*\nad x$, $w^*\ad g$ where $g\in PN(x)$, and $\lz
y_2,y_1,f_1,w^*,g,x\rz\cong P_6$. Thus, $w^*\nad f_1$ and hence
$w^*\ad \{f_2,z\}$. If $w^*$ is adjacent to a vertex of $M_U$,
then $w^*$ dominates $PN(x)$ by Proposition 1. Clearly, $x\in
M'_X$ for some component $M'$, $w^*\nad x$ and $w^*\ad g\in
PN(x)$. We have $\lz y_1,y_2,f_2,w^*,g,x\rz\cong P_6$. Therefore,
$w^*\nad \{u,u_1,u_2\}$. Now, $\lz
w^*,f_2,f_1,y_1,y,u_2\rz\not\cong P_6$ implies $y\nad u_2$, $\lz
w^*,z,y,y_1,y_2,u_2\rz\not\cong P_6$ implies $y\ad y_2$, $\lz
w^*,f_2,z,y_2,y,u_2,u_1\rz\not\cong F_1$ implies $y\nad u_1$, $\lz
w^*,z,y,y_2,u,u_1\rz\not\cong P_6$ implies $y\ad u$. Finally, $\lz
f_1,f_2,z,y,u,u_1\rz\cong P_6$, a contradiction. 

Let $M$ be of
type F4, and suppose that $w^*\ad u\in M_U$. By Proposition 1,
$w^*$ dominates $PN(x)$ for $x\in X$. Clearly, $x\not\in M_X$, and
hence $x\in M'_X$, where $M'$ is another component of $H-W$. We
have $w^*\ad g\in PN(x)$. By Lemma \ref{l1}, $x\ad x'\in M'_X$.
Since $N(w^*)\cap X\subseteq M_X$, we obtain $w^*\nad \{x,x'\}$.
If $w^*\ad v\in M_X$ and $y=PN(v)\cap I$, then $y\nad w^*$. We
obtain $\lz y,v,w^*,g,x,x'\rz\cong P_6$. Thus, $w^*\nad M_U$. If
$w^*\ad \{f_1,f_2\}$, then $\lz
w^*,f_1,f_2,y_1,y_2,u_1,u_2\rz\cong F_1$, and if $w^*\ad
\{f_1,z\}$, then $\lz u_2,y_2,y_1,f_1,w^*,z\rz\cong P_6$. Thus,
$w^*\nad f_1$ and $w^*\ad \{f_2,z\}$. Recall that $y\ad u_1$ and
$y\nad y_1$. We have $\lz w^*,z,y,u_1,y_1,f_1\rz\cong P_6$, a
contradiction. 

Let $M$ be of type F5. Since $|N(w^*)\cap M_X|\ge
2$, the vertex $w^*$ is adjacent to some vertex $z\in M_Z$. We
obtain $\lz u_1,y_1,y_2,f_2,z,w^*\rz\cong P_6$, a contradiction.

Let $M$ have type G1 or G2, and suppose that $w^*\ad u\in M_U$. By
Proposition 1, $w^*$ dominates $PN(x)$ for $x\in M'_X$, where $M'$
is another component of $H-W$. The vertex $w^*$ is adjacent to a
vertex from $\{f_1,f_2\}$, say $w^*\ad f_1$. We obtain $\lz
y_2,y_1,f_1,w^*,g,x\rz\cong P_6$, where $g\in PN(x)$. Thus,
$w^*\nad M_U$. We have $w^*\ad f_2$, for otherwise $w^*\ad f_3$
and $\lz u_1,y_1,y_2,f_2,f_3,w^*\rz\cong P_6$. Also, $w^*\nad
f_1$, for otherwise $\lz w^*,f_1,f_2,y_1,y_2,u_1,u_2\rz\cong F_1$.
Since $|X^*|\ge 2$, we obtain $w^*\ad f_3$. If $M$ is of type G1,
then $u_1\ad y_3$, for otherwise $\lz
u_1,y_1,f_1,f_2,f_3,y_3\rz\cong P_6$. Now $\lz
w^*,f_3,y_3,u_1,y_1,f_1\rz\cong P_6$, a contradiction. 

Let $M$ be
of type G2 and suppose that $y_1\nad y_3$. We have $u_1\ad y_3$,
for otherwise $\lz u_1,y_1,f_1,f_2,f_3,y_3\rz\cong P_6$. We obtain
$\lz w^*,f_3,y_3,u_1,y_1,f_1\rz\cong P_6$, a contradiction. Thus,
$y_1\ad y_3$. Furthermore, $u_2\nad y_3$, for otherwise $\lz
w^*,f_2,f_1,y_1,y_3,u_2\rz\cong P_6$. Also, $y_2\ad y_3$, for
otherwise $\lz w^*,f_3,y_3,y_1,y_2,u_1\rz\cong P_6$. Since
$\{y_1,y_2\}$ does not dominate $M_U$, there is a $u_3\in M_U\cap
N(y_3)$ such that $u_3\nad \{y_1,y_2\}$. Since $\lz
u_2,u_3,y_3,f_3,f_2,f_1\rz\not\cong P_6$, we obtain $u_2\nad u_3$.
Now $\lz w^*,f_2,f_3,y_2,y_3,u_2,u_3\rz\cong F_1$, a
contradiction. 

Let $M$ have type G3, and $w^*\ad \{f_1,f_3\},
w\nad f_2$. Assume that there is a $u\in N_1\cup N_2$ with $u\nad
w^*$. If $u\in N_1$, then $\lz u,y_1,y_2,f_2,f_3,w^*\rz\cong P_6$.
If $u\in N_2$, then $\lz u,y_2,y_1,f_1,w^*,f_3\rz\cong P_6$. Hence
$w^*$ dominates $N_1\cup N_2$. By Proposition 1, $w^*$ dominates
$PN(x)$ for some $x\in X$. If $x\in M'_X$ where $M'$ is another
component of $H-W$, then $w^*\ad g$ for $g\in PN(x)$ and $\lz
y_1,y_2,f_2,f_3,w^*,g\rz\cong P_6$. Since $w^*\nad \{y_1,y_2\}$,
$w^*$ dominates $PN(f_3)$. We know that $\{y_1,y_2\}$ dominates
$M_U$, and hence $M_U=N_1\cup N_2\cup N_{1,2}$. By the hypothesis,
$w^*$ cannot dominate $N_{1,2}$, so there is a $u\in N_{1,2}$
such that $w^*\nad u$. By another hypothesis, $u\nad u_1$. We
obtain $\lz u,y_1,u_1,w^*,f_3,f_2\rz\cong P_6$. 

Now let $M$ be of
type G4, and suppose that $w^*\ad f_1$. If $w^*\ad f_2$, then
$w^*$ is adjacent to $u_1$ or $u_2$, for otherwise $\lz
w^*,f_1,f_2,y_1,y_2,u_1,u_2\rz\cong F_1$. If $w^*\ad f_3$, then
$w^*\ad u_2$, for otherwise $\lz u_2,y_2,y_1,f_1,w^*,f_3\rz\cong
P_6$. Therefore, by Proposition 1, $w^*$ dominates $PN(x)$ for
$x\in M'_X$, where $M'$ is another component of $H-W$. We have
$w^*\nad x$, $w^*\ad g$ where $g\in PN(x)$, and $\lz
y_2,y_1,f_1,w^*,g,x\rz\cong P_6$. Thus, $w^*\nad f_1$ and hence
$w^*\ad \{f_2,f_3\}$. If $w^*$ is adjacent to a vertex of $M_U$,
then $w^*$ dominates $PN(x)$ by Proposition 1. Clearly, $x\in
M'_X$ for some component $M'$, $w^*\nad x$ and $w^*\ad g\in
PN(x)$. We have $\lz y_1,y_2,f_2,w^*,g,x\rz\cong P_6$. Therefore,
$w^*\nad \{u,u_1,u_2\}$. 

Suppose that $y_1\nad y_3$. Since
$\{y_1,y_2\}$ dominates $PN(f_3)$, we obtain $y_2\ad y_3$ and $\lz
f_1,y_1,y_2,y_3,f_3,w^*\rz\cong P_6$. Hence $y_1\ad y_3$. Now,
$\lz w^*,f_2,f_1,y_1,y_3,u_2\rz\not\cong P_6$ implies $y_3\nad
u_2$, $\lz w^*,f_3,y_3,y_1,y_2,u_2\rz\not\cong P_6$ implies
$y_3\ad y_2$, $\lz w^*,f_2,f_3,y_2,y_3,u_2,u_1\rz\not\cong F_1$
implies $y_3\nad u_1$, $\lz w^*,f_3,y_3,y_2,u,u_1\rz\not\cong P_6$
implies $y_3\ad u$. Finally, $\lz f_1,f_2,f_3,y_3,u,u_1\rz\cong
P_6$, a contradiction. \qed \vspace{-0.3cm}

\begin{lem}
\label{l11}
The set $X^*$  contains only
vertices from components of type B6 and D25.
\end{lem}

\pf Suppose to the contrary that $X^*$ contains a vertex of a
component $M$ of type different from B6 and D25. By Lemma
\ref{l10}, $X^*$ contains a vertex of another component $L$. Since
$X^*\cap I=\es$, neither $L$ nor $M$ can have type B3 or C.
Moreover, by Lemma \ref{l9}, $L$ and $M$ cannot be of types
D1--D24. Let $L$ have type A4 and $x\in X^*\cap M_X$. Obviously,
$w^*$ is adjacent to $f_2\in L_F$ or $z\in L_Z$, and $w^*\nad
\{y_1,y_2,f_1\}$. If $w^*\ad f_2$, then $\lz
f_1,y_1,y_2,f_2,w^*,x\rz\cong P_6$, while if $w^*\ad z$, then $\lz
y_2,y_1,f_1,z,w^*,x\rz\cong P_6$. Both cases yield a
contradiction. Therefore, neither $L$ nor $M$ may have type A4.

Suppose that one of the components $L$ and $M$, say $L$, is of
type B2, thus there are a component $L'$ of type B3 and a vertex
$w\in W$ such that $w\ad \{f,f'\}$ and $w\nad \{z,z'\}$, where
$f',z'\in L'_X$. Since $w,f,f',z'\in I$, we obtain $w^*\nad
\{w,f,f',z'\}$. We have $\lz w^*,z,f,w,f',z'\rz\cong P_6$. Hence,
neither $L$ nor $M$ can be of type B2. 

Assume now that $L$ has
type B4, thus there are a component $L'$ of type D1 and $w\in W$
such that $w\ad \{f,f_1,f_2\}$ and $w\nad \{z,y_1,y_2\}$, where
$f,z\in L_X$, $\{f_1,f_2\}=L'_X$ and $y_1,y_2\in L'_{PN}$. Since
$w,f,f_1,f_2\in I$, we have $w^*\nad \{w,f,f_1,f_2\}$. We obtain
$w^*\ad y_1$, for otherwise $\lz w^*,z,f,w,f_1,y_1\rz\cong P_6$.
Now $\lz f,z,w^*,y_1,f_1,f_2\rz\cong P_6$. Thus, neither $L$ nor
$M$ can have type B4. 

Suppose that $L$ has type B5, thus there is
the corresponding component $L'$ of type D2. We have $\lz
w^*,z,f,w,f_1,f_2\rz\not\cong P_6$, and hence $w^*\ad f_2$ where
$z,f\in L_X$, $f_1,f_2\in L'_F$. Now $\lz
f,z,w^*,f_2,f_1,y_1\rz\cong P_6$. Hence neither $L$ nor $M$ may
have type B5.

Let $L$ have type E  and $x\in X^*\cap M_X$.
If $w^*\ad z_1$, then $\lz y_2,f_2,f_1,z_1,w^*,x\rz\cong P_6$,
and if $w^*\ad z_2$, then $\lz y_1,f_1,f_2,z_2,w^*,x\rz\cong P_6$.
Hence, neither $L$ nor $M$ may have type E.
Thus, the components $L$ and $M$ may have any of the types A1--A3,
B1, F1--F5, G1--G5, or H, and also $L$ may have type B6 or D25.

The component $L$ contains vertices $a,b$ such that
$\lz a,b,w^*\rz_H\cong P_3$ and $a,w^*$ are its endvertices.
Indeed, $w^*$ is adjacent to $x\in L_X$.
On the other hand, for each of the above type of $L$, $I\cap V(L)\not=\es$
and $w^*\nad I$.
Since $L$ is a connected graph, we conclude that the above vertices
$a,b$ always exist.

It remains to consider the following cases. Let $M$ be of type A1.
We have $\lz a,b,w^*,f_2,$
$z,f_1\rz\cong P_6$, a contradiction.
Suppose that $M$ has type A2 or A3. We know that $w^*\ad z$ and
$w^*\nad \{f_1,f_2\}$. We have $w^*$ dominates $PN(f_2)$, for
otherwise there is a $g\in PN(f_2)$ such that $g\nad w^*$ and $\lz
g,f_2,z,w^*,b,a\rz\cong P_6$. Since $w^*$ cannot dominate $M_U\cup
PN(f_2)$, there is a $u\in M_U$ such that $u\nad w^*$. 

If $M$ is
of type A2, then $w^*\nad p$ and $p$ dominates $M_U$. We obtain
$\lz a,b,w^*,z,p,u\rz\cong P_6$. 

Let $M$ be of type A3, thus there
is a $w\in W$ such that $w$ dominates $M_U$ and $w\nad
\{f_1,y_1,y_2\}$. We have $w\ad z$, for otherwise $\lz
w,u_2,y_2,y_1,f_1,z\rz\cong P_6$. Since $w\in I$, we have $w\nad
w^*$. If $w$ is adjacent to $a$ or $b$, say $w\ad a$, then $\lz
f_1,y_1,y_2,u_2,w,a\rz\cong P_6$. Hence $w\nad \{a,b\}$. We obtain
$\lz a,b,w^*,z,w,u\rz\cong P_6$, a contradiction. 

Now let $M$ have
type B1. We know that $w^*\ad z$ and $w^*\nad \{f,y\}$. If $w^*\ad
p$, then $\lz a,b,w^*,p,y,f\rz\cong P_6$. Hence $w^*\nad p$ and
$\lz a,b,w^*,z,p,y\rz\cong P_6$. 

Let $M$ have type F1. We have
$\lz a,b,w^*,f_1,f_2,y_2\rz\cong P_6$ if $w^*\ad f_1$, and $\lz
a,b,w^*,z,$
$f_2,y_2\rz\cong P_6$ if $w^*\ad z$. 
Let $M$ be of type
F2 or F4. If $w^*\ad f_2$, then $\lz a,b,w^*,f_2,y_2,y_1\rz\cong
P_6$, and if $w^*\nad f_2$, then $\lz a,b,w^*,x,f_2,y_2\rz\cong
P_6$, where $x=f_1$ or $z$ and $x\ad w^*$. If $M$ is of type F3,
then $\lz a,b,w^*,z,f_2,f_1\rz\cong P_6$. If $M$ has type F5, then
$w^*\ad x\in M_X-f_2$. We have $\lz a,b,w^*,x,f_2,y_2\rz\cong
P_6$. 

Let $M$ have type G1, G2 or G4. If $w^*\ad f_2$, then $\lz
a,b,w^*,f_2,y_2,y_1\rz\cong P_6$, and if $w^*\nad f_2$, then $\lz
a,b,w^*,x,f_2,y_2\rz\cong P_6$, where $x=f_1$ or $f_2$ and $x\ad
w^*$. Let $M$ be of type G3. We have $\lz
a,b,w^*,x,f_2,y_2\rz\cong P_6$ where $x=f_1$ or $f_3$. If $M$ has
type G5, then $\lz a,b,w^*,f_3,f_2,f_1\rz\cong P_6$. 

Let $M$ be of
type H. We have $w^*\ad f_2$. Also, $w^*\ad y_1\in PN(f_1)$, for
otherwise $\lz a,b,w^*,f_2,f_1,y_1\rz\cong P_6$. Now $\lz
a,b,w^*,y_1,f_1,f_3\rz\cong P_6$. All cases yield a contradiction.
\qed

Now we are ready to deduce a final contradiction.
By Lemmas \ref{l8}, \ref{l10} and \ref{l11}, the set $X^*$ contains
a vertex of a component $L$ of type B6 and a vertex of a component
$L'$ having type B6 or D25.
Let $L'$ be of type B6 and $y',f',z'$ be vertices defined in this type.
We have $w^*\ad \{f,f'\}$ and $w^*\nad \{y,z,y',z'\}$.
The vertex $w^*$ dominates $PN(z)$, for otherwise
$\lz z',f',w^*,f,z,g\rz\cong P_6$ where $g\in PN(z)$ and $g\nad w^*$.
Analogously, $w^*$ dominates $PN(z')$.
Since $\lz y,f,w^*,f',y'\rz\cong P_5$, we easily deduce that $w^*$
dominates $L_U\cup L'_U$.
Therefore, $L$ is a component of type B2, a contradiction.

Now let $L'$ have type D25 and $f_1,f_2,y_1,y_2,u_1,u_2$ be chosen
as in the definition of this type. We have $w^*\ad f$ and $w^*\nad
\{y,z\}$. We know that one of the sets $\{y_1,y_2\}$,
$\{u_1,f_2\}$ or $\{u_2,f_1\}$ dominates $L'$ and belongs to $I$.
Let $\{y_1,y_2\}\subset I$ and w.l.o.g. $w^*\ad f_1$. We have
$w^*\ad f_2$, for otherwise $\lz z,f,w^*,f_1,f_2,y_2\rz\cong P_6$.
If $g\in PN(z)$ and $g\nad w^*$, then $\lz
g,z,f,w^*,f_1,y_1\rz\cong P_6$. Hence $w^*$ dominates $PN(z)$. If
$u\in L_U$ and $u\nad w^*$, then $\lz u,y,f,w^*,f_1,y_1\rz\cong
P_6$. Therefore, $w^*$ dominates $L_U$. 
Suppose that there is a
$u\in L'_U$ such that $u\nad w^*$. Clearly, $u\ad y_1$ or $u\ad
y_2$, say $u\ad y_1$. We have $\lz z,f,w^*,f_1,y_1,u\rz\cong P_6$.
Thus, $w^*$ dominates $L'_U$. We obtain that $L$ has type B4, a
contradiction. 

Finally, suppose that $\{u_1,f_2\}\subset I$ and $w^*\ad
f_1$. It is easy to show that $w^*$ dominates $L_U\cup PN(z)\cup
L'_{PN}$. Let $u\in N'_2$ and $w^*\nad u_2$. Since $\{u_1,f_2\}$
dominates $L'$, we obtain $u_1\ad u$. Now $\lz
z,f,w^*,y_2,u,u_1\rz\cong P_6$. Hence $w^*$ dominates $N'_2$.
Thus, $L$ must be of type B5, a contradiction. The symmetrical
case $\{u_2,f_1\}\subset I$ yields the same contradiction. The
proof of Theorem \ref{t1} is complete. \qed

We complete the paper with the following open problem:

\begin{prob}
Characterize $P_7$-free irredundance perfect graphs.
\end{prob}

\noindent
{\bf Acknowledgements.} PS was partially supported by NSF grant 2415564.

\end{document}